\documentclass[12pt]{amsart}
\usepackage[margin=1in]{geometry}
\usepackage[usenames,dvipsnames]{xcolor}
\definecolor{brightcerulean}{rgb}{0.11, 0.67, 0.84}
\usepackage[pagebackref=true, colorlinks=false, allcolors=brightcerulean, allbordercolors=brightcerulean,
pdfauthor={J. S. Balakrishnan, F. Bianchi, N. Dogra},
pdftitle={$p$-adic Elliptic Polylogarithms and Cubic Chabauty}
]{hyperref}
\usepackage{amsfonts, amsmath, amsthm, amssymb,tikz-cd}
\usepackage{comment}
\usepackage{colonequals}
\usepackage{appendix}
\usepackage{mathrsfs}

\usepackage{tikz}
\usetikzlibrary{matrix,arrows}

\DeclareMathOperator{\Ext}{Ext}
\DeclareMathOperator{\Res}{Res}
\DeclareMathOperator{\Hom}{Hom}
\DeclareMathOperator{\loc}{loc}

\DeclareMathOperator{\GL}{GL}

\DeclareMathOperator{\sym}{Sym}
\DeclareMathOperator{\ext}{Ext}

\DeclareMathOperator{\Sel}{Sel}
\DeclareMathOperator{\dR}{\textrm{dR}}
\DeclareMathOperator{\rk}{rk}

\newcommand{\Ker}{\mathrm{Ker}}

\newcommand{\Mat}{\mathrm{Mat}}
\newcommand{\Q}{\mathbb{Q}}
\newcommand{\fil}{\mathrm{fil}}
\newcommand{\F}{\mathbb{F}}
\newcommand{\End}{\mathrm{End}}
\newcommand{\gr}{\mathrm{gr}}
\newcommand{\et}{\mathrm{et}}
\newcommand{\ad}{\mathrm{ad}}

\newcommand{\Z}{\mathbb{Z}}

\DeclareMathOperator{\tors}{tors}

\DeclareMathOperator{\Gal}{Gal}
\DeclareMathOperator{\Sym}{Sym}
\renewcommand{\a}{\omega_0}
\renewcommand{\b}{\omega_1}
\newcommand{\R}{\mathbb{R}}

\theoremstyle{plain}
\newtheorem{theorem}{Theorem}[section]
\newtheorem{lemma}[theorem]{Lemma}
\newtheorem{proposition}[theorem]{Proposition}

\newtheorem{Conjecture}[theorem]{Conjecture}

\theoremstyle{definition}
\newtheorem{definition}[theorem]{Definition}
\newtheorem{example}[theorem]{Example}

\theoremstyle{remark}
\newtheorem{Remark}[theorem]{Remark}
\begin{document}

\begin{abstract} 
The Chabauty--Coleman--Kim method, under favourable circumstances, describes the set of integral points of a hyperelliptic curve inside the $p$-adic zeroes of certain transcendental functions. For an elliptic curve of Mordell--Weil rank one, the Chabauty--Coleman--Kim set in depth 2 is given by the zeroes of a (finite union of) quadratic polynomial(s) in the $p$-adic logarithm of the elliptic curve and the local $p$-adic height at $p$. Here, we give an explicit formula for a finite set containing the Chabauty--Coleman--Kim set in depth 3 for an elliptic curve of rank at most 2 under an assumption on non-vanishing of a special value of a $p$-adic $L$-function. The finite set is given by the zeroes of a polynomial in $p$-adic elliptic polylogarithms. We use these formulas to verify new instances of Kim's conjecture.
\end{abstract}

\title{$p$-adic Elliptic Polylogarithms and Cubic Chabauty}

\author[J.S. Balakrishnan]{Jennifer S. Balakrishnan}
\address{Jennifer S. Balakrishnan, Department of Mathematics and Statistics, Boston University, 665 Commonwealth Avenue, Boston, MA 02215, USA}
\email{jbala@bu.edu}
\author[F. Bianchi]{Francesca Bianchi}
\address{Francesca Bianchi}
\author[N. Dogra]{Netan Dogra}
\address{Netan Dogra, Department of Mathematics, King's College London,  Strand, London, WC2R 2LS, UK}
\email{netan.dogra@kcl.ac.uk}

\date{\today}
\maketitle
\section{Introduction}

Let $S$ be a finite set of primes, $X/\Z _S$ a smooth hyperbolic curve, and $p$ a prime not in $S$. One of the main aims of Kim's nonabelian Chabauty programme \cite{kim:siegel, kim:chabauty} is to prove finiteness of the depth $n$ nonabelian Chabauty set $X(\Z_p)_{S,n}$, which is achieved through a comparison of dimensions of local and global Selmer varieties. Since the set $X(\Z_p)_{S,n}$ contains $X(\Z _S )$, it is of interest to explicitly compute the set $X(\Z_p)_{S,n}$, after which, possibly in combination with other Diophantine techniques, one may recover $X(\Z _S )$.

The motivation behind this starts with the classical Chabauty--Coleman method, which can be viewed as the depth 1 case of Kim's programme. Here, for a smooth projective curve $X$ of genus $g$ and Jacobian rank $r$, if the rank $r$ is less than $g$, one has finiteness of the Chabauty--Coleman set $X(\Z_p)_1$ \cite{coleman:chabauty}, where $X(\Z_p)_n \colonequals X(\Z_p)_{\emptyset,n}$.  The Chabauty--Coleman set is described by the vanishing of (at least one) Coleman integrals of holomorphic one-forms.   

In depth 2, the first nonabelian case, if the rank $r$ is less than  $\rk NS(J) + g -1$, then the quadratic Chabauty set $X(\Z_p)_2$ is finite \cite{BD18}. Kim \cite{kim:chabauty} showed that the Bloch--Kato conjecture implies that $X(\Z_p)_n$ is finite for $n$ sufficiently large and made the following conjecture \cite{balakrishnan2012non} :
\begin{Conjecture}[Kim] For $n \gg 0$, we have $X(\Z_p)_n = X(\Z).$\end{Conjecture}

The aim of the present work is to study this conjecture in the case where $n=3$ and the curve $X$ is an elliptic curve $E$ with the origin removed. When the Mordell--Weil rank $r$ of $E$ is zero, Chabauty--Coleman gives finiteness of $X(\Z _p )_1 $, though the set may have further \emph{mock rational points}, those points that are not in $X(\Z)$.  When $r=1$, Kim \cite{kim:rank1}  showed that $X(\Z _p )_2 $ is finite (see also \cite{BKK11, balakrishnan-besser15}). Moreover, $X(\Z _p )_2 $ admits a simple description as the zeroes of a quadratic polynomial in $h_p (z)$, the local height at $p$ of the $p$-adic height function and $\int \omega$, the logarithm on $E$. This relation may be written as saying that $X(\Z _p )_2$ is a union of the zeroes of the functions
\begin{equation}\label{eqn:det_ht}
z\mapsto \det \left( \begin{array}{cc} h_p (z)+\alpha & (\int ^z \omega )^2 \\ h(P) & (\int ^P \omega )^2 \\ \end{array} \right),
\end{equation}
where $P$ is a point of infinite order, $\alpha $ is determined by congruence conditions at primes of bad reduction, and $h$ is the global $p$-adic height function.

In the present work we consider $X(\Z _p )_3 $ in the case when $r$ is at most 2. The problem of computing $X(\Z _p )_{S,3} $ when $X(\Z _p )_{S,2} $ need not be finite was also considered by Corwin \cite{corwin2021explicit}.  
As explained in loc. cit., when $r=2$, Kato's theorem implies that finiteness of $X(\Z _p )_3$ is related to non-vanishing of a special value of the $p$-adic $L$-function associated to $E$.  In this paper, we will explain how to use $p$-adic iterated integrals to compute functions vanishing on $X(\Z _p )_3$.

The functions describing $X(\Z _p )_{S,3}$ will be polynomials in $p$-adic elliptic polylogarithms in the sense of Bannai {et al.} \cite{bannai-kobayashi-tsuji,bannai-kings}. The formulas for $X(\Z _p )_{S,3}$ in terms of $p$-adic elliptic polylogarithms are closely related to the elliptic Zagier formulas of Goncharov and Levin \cite{goncharov-levin} for the classical special value $L(E,2)$ in terms of elliptic dilogarithms. Roughly speaking, in both cases, one wants to take a suitable linear combination of ($p$-adic) elliptic dilogarithms at rational points of $E$ in order to construct an element of $K_2 (E)$ (or its Galois cohomological avatar $H^1 _f (\Q ,V_p (E)(1))$).

In Lemma \ref{lemma:main}, we give a formula for $X(\Z _p )_{3}$ when $E$ has rank 2 and under some conditions on the reduction types of $E$. Analogous to the depth 2 case, the formula is given as the zeroes of a determinant of matrices of a collection of polynomials in $p$-adic elliptic polylogarithms evaluated at suitably independent points.
As in the case of linear and quadratic Chabauty, to construct our formulas for integral points we first need to find enough points to determine the nontrivial linear relations a finite set of $p$-adic analytic functions satisfy on integral points. 

In Section \ref{sec:CCK} we recall the definition of $X(\Z _p )_{S,n}$, and the definition of Beilinson--Levin's elliptic polylogarithm sheaf. In Section \ref{sec:realizations}, we recall some formulas for the de Rham and crystalline realisations of the elliptic polylogarithm sheaf. In Section \ref{sec:linearizations}, we explain the relation between elliptic polylogarithms and nonabelian Chabauty, and how a suitable linearisation procedure allows one to find formulas for $X(\Z _p )_{S,n}$ (defined in terms of nonabelian Galois cohomology) in terms of abelian Galois cohomology. In Section \ref{sec:triple}, we explain how to algorithmically compute the $p$-adic elliptic polylogarithm functions constructed, and in Section \ref{sec:examples} we use this to verify new instances of Kim's conjecture.

\subsection*{Acknowledgements}
We would like to thank Minhyong Kim for initially suggesting this project over a decade ago. We are grateful to Andrew Sutherland for giving access to \texttt{chatelet.mit.edu} and Ariel Pacetti and Fernando Rodriguez Villegas for making available their PARI/GP code. Balakrishnan was supported by NSF grants DMS-1945452 CAREER and DMS-2502687, Bianchi was supported by NWO grant VI.Vidi.192.106, and Dogra was supported by a Royal Society University Research Fellowship.

\section{The Chabauty--Coleman--Kim method}\label{sec:CCK}
We first recall the general set-up of the Chabauty--Coleman--Kim method. Let $S$ be a finite set of primes and suppose we are given a curve $X/\Q $. Let $\overline{X}$ be a smooth compactification and $\overline{\mathcal{X}}$ a minimal regular model over $\Z _S$. Let $\mathcal{X}$ denote the complement of the closure of $\overline{X}-X$ in $\overline{\mathcal{X}}$. To ease notation, for a $\Z _S$-algebra $R$ we will sometimes write $X(R)$ instead of $\mathcal{X}(R)$. Let $U(X,b)$ be a Galois-stable finite-dimensional quotient of its $\Q _p $-unipotent fundamental group. For any other geometric point $z$ of $X$, let
\[
P (X;b,z):=\pi _1 ^{\et }(X_{\overline{\Q }};b,z)\times _{\pi _1 ^{\et }(X_{\overline{\Q }},b)}U (X,b).
\]
If $z$ is defined over a field $K/\Q $, then $P (X,b)$ has the structure of a $\Gal (\overline{K}|K)$-equivariant $U (X,b)$-torsor, defining a map
\[
j_{U,K}:X(K)\to H^1 (\Gal (\overline{K}|K),U (X,b)).
\]
Define $j_{U,\Z _{\ell }}$ to be the map $j_{U,\Q _{\ell }}$ restricted to $X(\Z _{\ell })$. The maps $j_{U,K}$ are compatible as $K$ varies. Define $
X(\mathbb{A}_\Q )_{S,U}\subset \prod _{\ell \in S}X(\Q _{\ell })\times \prod _{\ell \notin S}X(\Z _{\ell })
$
to be the pre-image under $$\prod _{\ell \in S}j_{U,\Q_{\ell }}\times \prod _{\ell \notin S}j_{U,\Z _{\ell }}$$ of the subset 
\[
\loc H^1 (\Gal (\overline{\Q }|\Q ),U (X,b))\subset \prod H^1 (G_{\Q _v },U (X,b)).
\]
Define $X(\Z _p )_{S,U}\subset X(\Z _p )$ to be the projection of $X(\mathbb{A}_\Q )_{S,U}$ to $X(\Z _p )$. In words, $X(\mathbb{A}_{\Q })_{S,n}$ is the set of $S$-integral adelic points whose $U (X,b)$-cohomology classes come from global cohomology classes, and $X(\Z _p )_{S,U}$ is the set of $\Z _p $-points which extend to $S$-integral adelic points with this property. When $S$ is empty, we write this simply as $X(\Z _p )_U$. When $U$ is the maximal $n$-unipotent quotient of the $\Q _p$-unipotent fundamental group, we write $X(\Z _p )_{S,U}$ and $X(\Z _p )_U$ as $X(\Z _p )_{S,n}$ and $X(\Z _p )_n$ respectively.

\subsection{The elliptic polylogarithm sheaf}\label{sec:pi1E}
We now specialise to the case of interest for this paper, where $X$ is an elliptic curve $E$ minus the origin 
and $U_n (z)$ is the maximal $n$-unipotent quotient of the maximal pro-$p$ metabelian quotient $U_\infty $ of the \'etale fundamental group of $X_{\overline{K}}$:
\[
U_\infty (z):=\pi _1 ^{\et ,(p) }(X_{\overline{K}},z)/[[\pi _1 ^{\et ,(p) }(X_{\overline{K}},z),\pi _1 ^{\et ,(p) }(X_{\overline{K}},z)],[\pi _1 ^{\et ,(p) }(X_{\overline{K}},z),\pi _1 ^{\et ,(p) }(X_{\overline{K}},z)]].
\]
We define $L_n ^{\et}(z)$ to be the Lie algebra of $U_n (z)$, and define $L_{i,n}^{\et}(z):=\Ker (L_n ^{\et}(z)\to L_{i}^{\et}(z))$.
The pro-Lie algebra $(L_n ^{\et}(z))_n $ is the fibre at $z$ of a pro-system of Lie algebra objects 
$(L_n ^{\et} )_n $ in the category of $\Q _p $-local systems on $X_{\overline{\Q }}$, which can be identified with the (\'etale realisation of the) \textit{elliptic polylogarithm sheaf} first introduced by Beilinson and Levin \cite{beilinson1991elliptic}. Since $\pi _1 ^{\et }(X,z)$ is isomorphic, as a profinite group, to the profinite completion of a free group on $2$ generators, $U_n (z)$ and $L_n ^{\et}(z)$ have a rather explicit description (see e.g.  \cite{huber1999degeneration, beilinson1991elliptic, kings2013eisenstein}), which we now recall. 

Let $T_0$ denote the set of primes of bad reduction, and define $T:=T_0 \cup \{ p\},$ where $p$ 
is a prime of good reduction. Let $\omega $ denote the N\'eron differential, and let $v$ denote the corresponding tangent vector at the identity. We introduce the following notation. Let $K$ be a field of characteristic zero. For a $K$-vector space $W$, $\Sym^\bullet W$ denotes the symmetric algebra on $W$. Denote by $\Sym(W)_n $ the quotient of $\Sym^\bullet W$ by the ideal generated by $\Sym^{n+1}W$. Equivalently, 
$\Sym(W)_n $ is $\oplus _{i=0}^n \Sym^i W$, viewed as a nilpotent $K$-algebra. For each $n>0$, we have a natural 
exponential homomorphism
$W\hookrightarrow \Sym(W)_n ^\times $ 
sending $w$ to $e^w :=\sum _{i=0}^n  \frac{w^i }{i! }$. 
For a set $X$ and a function $ f:X\to W$
denote by $\Sym^i f$ the corresponding function
\begin{align*}
 X&\longrightarrow \Sym^i W\\
 x&\mapsto f(x)^i.
\end{align*}

Let $X_0 $ and $X_1$ be arbitrary generators of $L_n ^{\et } (z)$. 
Recall that by construction the Lie algebra $L_{1,n}^{\et}(z)$ is abelian, so the adjoint action of $L_n ^{\et}(z)$
 on $L_{1,n}^{\et}(z)$ factors through $V:=V_p (E)\simeq  L_1 ^{\et }(z)$. Hence $L_{1,n}^{\et}(z)$ defines a Galois-equivariant $\Sym^\bullet V$-module. 
 It is easy to see that $L_{1,n}^{\et}(z)$ is a free rank 1 $\Sym(V)_{n-2}$ module,
 generated by $[X_0 ,X_1 ]$. We shall use this to write elements of $L_n (z)$ as 
 \begin{equation}\nonumber
  W=W_0 +F(X_0 ,X_1 )[X_0 ,X_1 ]
 \end{equation}
where $F$ is a (commutative) polynomial in $X_0 $ and $X_1 $, $W_0$ is a linear combination of $X_0$ and $X_1 $, and $F(X_0 ,X_1 )$ is understood to be acting on $[X_0 ,X_1 ]$ 
via the adjoint action explained above. 

Whenever one has a (rational) tangential basepoint $v$ it induces a Galois-equivariant homomorphism
$\mathbb{Q}_p (1)\to U(v)$
and hence by composition one obtains
$\mathbb{Q}_p (1)\to U_n (v)$, see \cite{deligne1989groupe}. By the standard presentation of the topological fundamental group of an elliptic curve minus the origin, the image of the tangential basepoint homomorphism is trivial when $n=1$ and nontrivial when $n>1$. This allows us to describe the structure of $L_{1,n}^{\et}(v)$ completely:
\begin{lemma}\label{lemma:its_sym}
 There is a Galois-equivariant isomorphism
 \begin{equation}\nonumber
  L_{1,n}^{\et}(v)\simeq \oplus _{i=0}^{n-2}\Sym^i (V)(1).\end{equation}
\end{lemma}
If we now consider the representation $L_n ^{\et} (v)$, the lemma above implies that via the above splitting, $L_n ^{\et} (v)$ defines, for all $1<i\leq n$, a class 
$\kappa ^v _i \in \Ext^1 (V,\Sym^{i-2}(V)(1))$. 
\begin{lemma} We have $\kappa ^v _i =0$ for all even $i$. \end{lemma}
\begin{proof}
The involution $(x,y)\mapsto (x,-y)$ on $E$  can be used to define an involution on $L_n ^{\et}(v)$, which on $L_{i,i+1}^{\et}(v)$ is given by 
multiplication by $(-1)^i$. This produces the required splitting. 
\end{proof}
By contrast, when $i$ is odd, $\kappa ^v _i$ will typically be nontrivial. The first non-trivial case, which is all we will need, is described in 
\cite[2.1.6]{beilinson1991elliptic}. Let $\Delta \colonequals \Delta (E,\omega )\in \Q ^\times$ denote the discriminant of the elliptic curve $E$ relative to the differential $\omega $.
\begin{lemma}\label{lemma:Delta}We have $\kappa ^v _3 =\frac{1}{12}\iota (\Delta)$, where $\iota $ is the composite map
\begin{equation}\nonumber
 \mathbb{Q}^\times \to H^1 (G_{\mathbb{Q},T},\mathbb{Q}_p (1))\to \Ext^1 (V,\Sym^2 (V)(1)).
\end{equation}
\end{lemma}
The final result we shall need regarding $L_n (v)$ is
\begin{lemma}\label{wildeshaus}
 For all $i\geq 3$, $\kappa ^v _i $ is in the image of $H^1 (G_T ,\Sym^{i-2}(V)(1))$ under
 \begin{equation}\nonumber
  H^1 (G_T ,\Sym^{i-2}(V)(1))\to \Ext^1 (V,\Sym^{i-3}(V)(1)).
 \end{equation}
\end{lemma}
\begin{proof}
\cite[Lemma 3.19]{wildeshaus}.
\end{proof}

Now we replace $v$ by an arbitrary basepoint $z$. The Galois representation $L_n ^{\et} (z)$ 
can be described in terms of $L_n ^{\et}(v)$ and the torsor $P_n (v,z)$, via the action of $U_n (v)$ on $L_n ^{\et}(v)$. This has a 
very simple description thanks to a special case of the Baker--Campbell--Hausdorff formula: 
\begin{lemma}\label{BCH}
Let $W=W_0 +F(X_0 ,X_1 )[X_0 ,X_1 ]$ be an element of $L_n ^{\et} (z)$, and $\gamma = \exp(C)$ an element of $U_n (z)$. Suppose 
\begin{equation}\nonumber
C=C_0 +G(X_0 ,X_1 )[X_0 ,X_1 ]
\end{equation}
Then
\begin{equation}\nonumber
\gamma W\gamma ^{-1}= \sum _{i=0}^n \frac{1}{i!}\ad (C_0 )^i (W_0 )+(e^{C_0 }F(X_0 ,X_1 ))[X_0 ,X_1 ]-(W_0 .G(X_0 ,X_1 ))[X_0 ,X_1 ].
\end{equation}
\end{lemma}
 In particular, for the subrepresentation $L_{1,n} ^{\et}(v)^{(c)}$, we obtain the following corollary:
\begin{proposition}\label{logarithm}
There is a Galois-equivariant isomorphism
\begin{equation}\nonumber
L_{1,n} ^{\et}(v)^{(c)}\simeq \Sym^{n-2}(\kappa (c)),
\end{equation}
where $\kappa (c)$ denotes the extension of $\mathbb{Q}_p $ by $V$ corresponding to the image of $c$ in $H^1 (G_T ,V)$.
\end{proposition}
\begin{proof}
The conjugation action of $U_n (v) $ on $L_{1,n}^{\et}(v)$ factors through $V$. By Hadamard's formula for the conjugation action of a group on its Lie algebra, we deduce that if $c:G_T \to V$ is a cocycle, then the action on $L_{1,n}^{\et}(v)^{(c)}$ is given by
\[
g(v^c )=e^{c(g)}\cdot g(v)
\]
for $v\in L_{1,n}^{\et}(v)\simeq \oplus _{i=0}^{n-2}\Sym ^i (V)(1)$, giving an explicit isomorphism with $\Sym ^n (\kappa (c))$.
\end{proof}

\subsection{$p$-adic elliptic polylogarithms when $\ell \neq p$}
We now describe the localisation of the extension classes $[L_n ^{\et}(z)]$ at primes $v\in T$. 
First let $v$ be a prime in $T_0$. Since $H^1 (G_v ,V)=H^0 (G_v ,V)=0$, there are canonical isomorphisms of $G_v $-representations
\begin{equation}\nonumber
 L_{1,n}^{\et}(z)\simeq \Sym^{n-2}(V)(1)
\end{equation}
and a canonical isomorphism of Galois cohomology spaces
\begin{equation}\nonumber
 H^1 (G_v ,U_n )\simeq H^1 (G_v ,U_{1,n})\simeq \oplus _{i=0}^{n-1} H^1 (G_v ,\gr _i U_n ),
\end{equation}
where $\gr _i$ denotes the $i$th graded piece with respect to the central series filtration. 
If $v$ is a prime of split multiplicative reduction, this gives isomorphisms
\begin{equation}\nonumber
 H^1 (G_v ,U_n )\simeq H^1 (G_v ,\mathbb{Q}_p (1))^{\oplus n-1},
\end{equation}
\begin{equation}\nonumber
\Ext^1 _{G_v }(V,L_{1,n}^{\et}(z))\simeq H^1 (G_v ,\mathbb{Q}_p (1))^{\oplus n-2}.
\end{equation}
The specific cohomology classes obtained may be described explicitly in terms of Bernoulli polynomials \cite{schappacher1991boundary}. Namely, let $B_3 (x)=x^3 -\frac{3}{2}x^2 +\frac{1}{2}x$ denote the third Bernoulli polynomial. Then we have the following result:
\begin{theorem}[Schappacher--Scholl]\label{thm:SS}
If $E$ has split multiplicative reduction at $v$ and the special fibre $\mathcal{E}_v$ of the minimal regular model is an $N$-gon, then for a $v$-integral point $P$, the class of $P$ in $\Ext ^1 (V,V(1))\simeq H^1 (\Q _v ,\Q _p (1))\simeq \Q _p $ is given by $B_3 (m(P)/N)$, where $m(P)\in \{ 0,\ldots, N-1\}$ is such that $P$ reduces to the $m$th component of the special fibre of $\mathcal{E}_v$.
\end{theorem}

\section{Crystalline and de Rham realisations of elliptic polylogarithms}\label{sec:realizations}
In this subsection we describe the de Rham analogue of the Galois representation constructed above, following the approach of Bannai, Kobayashi and Tsuji \cite{bannai-kobayashi-tsuji}. More recently, Ben Moore gave a complete description of the Hodge filtration on the elliptic KZB equation in arbitrary depth \cite{moore2025rham} with a view to applications to nonabelian Chabauty. Earlier work of Beacom \cite{beacom} described an algorithm for computing the Hodge filtration on similar objects.  As we need to compare our normalisations with other ones used (in particular for Coleman integrals), we redo the Hodge filtration computations from scratch for the sake of completeness, but not with the aim of originality. 

Let $K$ be a number field, and let  $y^2 =x^3 +ax+b$ be a Weierstrass model for $E$ defined over $K$. As before, let $X=E-\{ 0 \}$.  Denote by $\mathcal{C}^{\dR}(X_K )$ the Tannakian category of unipotent flat connections on $X_K $. 
We first recall the de Rham realisation of the tangential basepoint functor. Let $\mathcal{C}^{\dR}_{\log }(E_K )$ denote the Tannakian 
category of unipotent flat connections on $E_K$ with log poles along $0$. By Deligne's canonical extension theorem \cite[Proposition 5.2]{deligne:DEs} the forgetful functor
\begin{equation}\nonumber
 \mathcal{C}^{\dR}_{\log }(E_K )\to \mathcal{C}^{\dR}(X_K )
\end{equation}
gives an equivalence of categories. Let $\mathcal{C}^{\dR}_{\log }(T_0 E )$ denote the category of connections on the tangent space of $E$ at $0$ with log singularities along 
the zero vector $0$. Let
\begin{equation}\nonumber
\Res : \mathcal{C}^{\dR}_{\log }(E_K )\to \mathcal{C}^{\dR}_{\log }(T_0 E )
\end{equation}
be the functor sending a connection $(\mathcal{V},\nabla )$ to $(O^* \mathcal{V},\Res _0 \nabla )$. Then for any nonzero tangent vector $t$ we obtain a Tannakian basepoint on 
$\mathcal{C}^{\dR}_{\log }(E_K )$ by composition 
\begin{equation}\nonumber
\mathcal{C}^{\dR}_{\log }(E_K )\to \mathcal{C}^{\dR}_{\log }(T_0 E ) \stackrel{t^* }{\longrightarrow } K,
\end{equation}
which we call the tangential basepoint functor associated to $t$, and also denote by $t^*$.
\subsection{Universal connections arising from the de Rham fundamental group}
Let $\omega _0 $ denote the differential
 $dx/2y$, and let $\omega _1 $ denote the differential of the second kind $xdx/2y$. Define $\kappa ^{\dR}$ to be the connection on $X$ with underlying bundle 
$$\mathcal{O}_X \oplus \mathcal{O}_X .W_0 \oplus \mathcal{O}_X .W_1 ,$$ with $\nabla $ defined by
\begin{align*} 1     &\mapsto -\omega _0 \otimes W_0 -\omega _1 \otimes W_1,\\
W_0 &\mapsto 0,\\
W_1 &\mapsto 0.
\end{align*}

Let $\widetilde{U}^{\dR}(x)$ denote the Tannaka fundamental group associated to the category $\mathcal{C}^{\dR}(x)$ and the fibre functor $x$. As before, let 
$\widetilde{L}^{\dR}(x)$ denote the associated pro-nilpotent Lie algebra. Let $\widetilde{L}_n ^{\dR}(x)$ denote the maximal $n$-step nilpotent quotient of 
$\widetilde{L}_n ^{\dR}(x)$, and define
\begin{align*}
 L_n ^{\dR}(x) &=\widetilde{L}_n ^{\dR}(x) /[[\widetilde{L}_n ^{\dR}(x) ,\widetilde{L}_n ^{\dR}(x) ],[\widetilde{L}_n ^{\dR}(x) ,\widetilde{L}_n ^{\dR}(x) ]] \\
 L_{i,n} ^{\dR}(x) &=\ker(L_n ^{\dR}(x)\to L_i ^{\dR}(x)).
\end{align*}
In analogy with the \'etale case, the $\widetilde{U}^{\dR}(x)$ action on $L_n ^{\dR}(x) $ and $L_{i,n}^{\dR}(x)$ defines flat connections $L_n ^{\dR}$ and $L_{i,n}^{\dR}$, and the fibres 
of these bundles at a basepoint $z$ are canonically identified with $L_n ^{\dR}(z)$ and $L_{i,n}^{\dR}(z)$ respectively. The connection $L_n ^{\dR}$ admits the following explicit description. Let $L_n$ be the free depth $n$ metabelian nilpotent Lie algebra on two generators $T_0 $ and $T_1$, as above. Define a connection $\nabla $ on the vector bundle $L_n ^{\dR}:=L_n \otimes \mathcal{O}_X$ by
\[
\nabla (v\otimes f)=v\otimes df -[T_0 ,v]\otimes f\omega _0 -[T_1 ,v]\otimes f\omega _1 .
\]
A tedious but straightforward calculation gives the following description of parallel transport on $L_n ^{\dR}$ for $n\geq 4$.
\begin{lemma}\label{lemma:tedious}
A general (formal) horizontal section of $L_4 ^{\dR}$ at $x\in E-0$ passing through $Y\in V \subset L_4$ is given by

\begin{align*}
&Y + a_0 [T_0,Y] + a_1 [T_1,Y]
 - \left(a_{01}-\frac{a_0 a_1}{2}\right) Y[T_0,T_1]- \left(a_{001}+\frac{a_0^2 a_1}{12}\right) YT_0[T_0,T_1] \\
&\quad + \left(a_{110}-\frac{7 a_0 a_1^2}{12}\right) YT_1[T_0,T_1]  + \frac{a_0^2}{2} \, T_0[T_0,Y]
 + \frac{a_1^2}{2} \, T_1[T_1,Y]
 + \frac{a_0 a_1}{2}\bigl(T_0[T_1,Y] + T_1[T_0,Y]\bigr) \\
&\quad + \frac{a_0}{6}(a_0T_0 + a_1T_1)^2 [T_0,Y]
 + \frac{a_1}{6}(a_0T_0 + a_1T_1)^2 [T_1,Y],
\end{align*}
where, for $i_1 ,\ldots ,i_k \in \{0,1\}$, we define
\[
a_i :=\int \omega _{i_1 }\ldots \omega _{i_k }
\]
to be the formal iterated integral.
\end{lemma}
\subsection{Canonical extensions and tangential basepoints}

Let $N_n \subset \Mat _n $ be the Lie algebra of lower diagonal nilpotent matrices, and let $U_n =1+N_n =\exp (N_n )\subset \GL _n$ be the group of lower diagonal unipotent matrices.
\begin{lemma}
Let $K$ be a field of characteristic zero. Given a lower triangular nilpotent matrix $\Lambda \in N_n (K(\! (w)\! ) )$, there is a unique $G\in 1+w^{-1}N_n (K[w^{-1}])$ such that $G^{-1}\Lambda G-G^{-1}dG$ has log singularities.
\end{lemma}

In particular, $L_n$ extends uniquely to a vector bundle with connection on $E$ with log singularities along $0$. We will also denote this vector bundle by $L_n ^{\dR}$. As before, $L_{1,n}^{\dR}$ will denote the kernel of the morphism of connections with log singularities $L_n ^{\dR}\to L_1 ^{\dR}$. We have an isomorphism of vector bundles with flat connection
\[
L_{1,n}^{\dR}\simeq \Sym ^{n-2}\kappa ^{\dR}.
\] 
$L_n ^{\dR}$ carries a descending filtration by sub-bundles, which is induced from the Hodge filtration on the de Rham fundamental group \cite{deligne1989groupe}, and lifts the Hodge filtration on the graded pieces $L_{i,i+1}^{\dR}$ for $i>0$ (identifying them with subobjects of $(H^1 _{\dR}(E/K)^* )^{\otimes (i-1)}(1)\otimes \mathcal{O}_X$) and $L_1 ^{\dR}$ (identifying it with $H^1 _{\dR}(E/K)^* \otimes \mathcal{O}_X $). The isomorphism of vector bundles with flat connection lifts to an isomorphism of filtered flat connections
\[
L_{1,n}^{\dR}\simeq \Sym ^{n-2}\kappa ^{\dR}(1).
\]
Here the filtration on $\kappa ^{\dR}$ is defined by 
\[
\kappa ^{\dR}=F ^{-1}\kappa ^{\dR}\supset F ^0 \kappa ^{\dR}=\mathcal{O}_X \oplus \mathcal{O}_X \cdot W_1 .
\]
We have the following analogue of Hadian's characterisation of the Hodge filtration on the universal enveloping algebra of the de Rham fundamental group \cite{hadian2011motivic}.
\begin{lemma}
There is a unique lift of the Hodge filtrations on $L_{1,n}$ and $V\otimes \mathcal{O}_E$ to a filtration on $L_n$ which extends to a filtration on $L_{n+1}$ extending that on $L_{1,n+1}$.
\end{lemma}
\begin{proof}
Recall that $F^0 L_{1,n}=0$. Hence the set of lifts of $F^0 V\otimes \mathcal{O}_E$ to $L_n ^{\dR}$ is a torsor under $H^0 (E,L_{1,n}^{\dR})$. Similarly the collection of lifts of $V\otimes \mathcal{O}_E$ to $L_n$ is a torsor under $H^0 (E,V^* \otimes L_{1,n}/F^0 )$. Hence the lemma follows from the fact that the maps
\[
H^0 (E,L_{1,n+1})\to H^0 (E,L_{1,n}) \qquad\textrm{and}\qquad
 H^0 (E,V^* \otimes L_{1,n+1}/F^0 )\to H^0 (E,V^* \otimes L_{1,n}/F^0 )
 \]
are zero.
\end{proof}

Let $t=-x/y$ be a local parameter for $E$ at the origin. Recall our basis for the de Rham cohomology of $E-\{ 0 \}$ is given by $\omega _0 =dx/2y $ and $\omega _1 =xdx/2y$. Define
\begin{equation}\nonumber
 w(t)=-\frac{1}{y}\equiv t^3 +at^7 +bat^9 +2a^2 t^{11} +(2ab +3a )t^{13} \pmod{t^{15}}. \end{equation}
Then
\begin{align*}
 x(t)& =t/w =t^{-2}-at^2 -bt^4 -a^2 t^6-3a t^{8} \pmod{t^{10}} \\
 \omega _0 (t) &=-\frac{1}{2}w(t)dx(t) = (1 +2a t^4 + 3a(a+b) t^6 +(15a +2ab +2a^2 )t^{8})dt \pmod{t^{10}} \\
\omega _1 (t) &\equiv (t^{-2}+at^2 +a(a+3b)t^4 +(15a +2ab)t^6 ) dt \pmod{t^8}.
\end{align*}
Then we see that we may extend $\kappa _{\dR}$ to a vector bundle with flat connection on $E$ as follows. Take $\kappa ^{\dR}_{\infty }$ to be a three-dimensional trivial vector bundle on a suitably small neighbourhood $U$ of the origin $0$, and define an isomorphism
\[
\kappa ^{\dR}_{\infty }|_{U-\{ 0 \} }\stackrel{\simeq }{\longrightarrow }\kappa ^{\dR}|_{U-\{ 0 \} }
\]
by $W_i \mapsto W_{i,\infty }$ for $i=1,2$ and $W_0 \mapsto W_{0,\infty } -t^{-1}W_{2,\infty }$. 
Define the connection on $\kappa _{\infty } ^{\dR}$ so that this is a morphism of connections, i.e.
\[
\nabla =d+\left( \begin{array}{ccc} 0 & 0 & 0 \\ -\omega _0+t^{-2}dt & 0 & 0 \\ -\omega _1 & 0 & 0 \end{array} \right) .
\]
with respect to the basis $W_{i,\infty }$.
The gauge transformation on $\kappa ^{\dR}$ is given by
\begin{equation}\nonumber
 1\mapsto 1-t^{-1}T_1.
\end{equation}
Hence the gauge transformation on $\Sym^n \kappa ^{\dR}$ is given by 
\begin{equation}\label{eqn:exp}
 v\mapsto e^{-t^{-1}W_1 }v.
\end{equation}
This similarly gives an extension of $L_{1,n}^{\dR}$ to a vector bundle with flat connection on $E$, via the identification with $\Sym ^{n-2}(\kappa _{\dR})$. Explicitly, we can take coordinates $T_{0,\infty }^i T_{1,\infty }^j [T_0 ,T_1 ]_{\infty }$ and a gauge transformation
\[
T_0 ^i T_1 ^j [T_0 ,T_1 ]\mapsto e^{-t^{-1}T_{1,\infty } }T_{0,\infty} ^i T_{1,\infty } ^j [T_0 ,T_1 ]_{\infty }.
\]
To extend this to $L_n ^{\dR}$, we define \[
L_{n,\infty }:=L_{1,n,\infty}^{\dR}\oplus \mathcal{O}_Y\cdot T_{0,\infty }\oplus \mathcal{O}_Y \cdot T_{1,\infty },
\] 
and look for $G_i \in L_{1,n}\otimes \mathcal{O}(U)$ such that
\[
\Lambda _{\infty }:=A^{-1}\Lambda A+A^{-1}dA\in \End (L_n \otimes \mathcal{O}_Y )\cdot \frac{dt}{t}
\]
where 
\[
A=\left( \begin{array}{ccc} 1 & 0 & 0 \\ 0 & 1 & 0  \\ G_1 & G_2 &  e^{-t^{-1}T_1 } \end{array}\right) .
\]
Although the choice of $A$ satisfying this condition is not unique, the  residue of $\Lambda _{\infty }$ is. To ease notation, we will sometimes write $G_i$ as a function valued in the polynomial ring in $T_0 $ and $T_1$.

\begin{lemma}\label{lemma:res_class} \;\hfill
\begin{enumerate} 
\item We may take
\begin{align*}
G_0 & =t^{-1}-\frac{1}{2}t^{-2}T_1  +\frac{-t^{-1}}{2}T_0 T_1 +\frac{1}{6}t^{-3}T_1 ^2 +\frac{5}{12}t^{-2}T_0 T_1 ^2 -\frac{1}{24}t^{-4}T_1 ^3 . \\
G_1 & =-\frac{1}{2}t^{-1} T_1 ^2 +\frac{5}{12}t^{-2}T_1 ^3.
\end{align*}
\item The residue of $L_6 ^{\dR}$ is given by $\ad (T)$, where $T\in L_4 $ is equal to
\[
[T_0 ,T_1 ]_{\infty }-\frac{1}{2}T_{0,\infty } T_{1,\infty } [T_0 ,T_1 ]_{\infty }.
\]
\item The lift of the Hodge filtration on $V\otimes \mathcal{O}_E =L_1 ^{\dR}$ to $L_n ^{\dR}$ is given by 
\begin{align*}
F^0 L_4 ^{\dR}\ni T_1 -\frac{1}{2}T_1 [T_0 ,T_1 ], \\
F^{-1}L_4 ^{\dR} \ni T_0 -\frac{1}{2}T_0 [T_0 ,T_1 ].
\end{align*}
\end{enumerate}
\end{lemma}

We deduce that, at infinity, with respect to this gauge transformation, the connection is given by
\begin{align*}
T_{0,\infty }&\mapsto e^{t^{-1}T_{1,\infty } }(\omega _1 -(\omega _0 T_{0,\infty } +\omega _1 T_{1,\infty } )G_0 +dG_0 )[T_0 ,T_1 ]_{\infty }, \\
T_{1,\infty }&\mapsto e^{t^{-1}T_{1,\infty } }(-\omega _0 -(\omega _0 T_{0,\infty } +\omega _1 T_{1,\infty } )G_1 +dG_1 )[T_0 ,T_1 ]_{\infty }.
\end{align*}
We may write this out (even more) explicitly as \small\begin{align*}
T_0 \mapsto & e^{t^{-1}T_{1 }}(\omega _1 -t^{-2}dt -t^{-1}\omega _0 T_{0 }+(-t^{-1}\omega _1 +t^{-3}dt)T_{1}  +\left(\frac{1}{2}\omega _0 t^{-2}+\frac{t^{-2}dt}{2}\right)T_0 T_1+\left(\frac{1}{2}\omega _1 t^{-2}-\frac{1}{2}t^{-4}dt\right )T_1 ^2) \\
& = \omega _1 -t^{-2}dt-t^{-1}\omega _0 T_0 +\frac{1}{2}(-\omega _0 t^{-2}+t^{-2}dt )T_0 T_1  \\
T_1 \mapsto & e^{t^{-1}T_1}\left(-\omega _0 +\frac{1}{2}t^{-2}dt T_1 ^2 \right) \\
& = -\omega _0 -t^{-1}\omega _0 T_1 -\frac{1}{2}t^{-2}(-\omega _0 +dt)T_1 ^2 .
\end{align*}
\normalsize
We deduce the following:
\begin{lemma}
For $z_1, z_2 $ away from the poles of $t$ (but possibly tangential basepoints at infinity), the unique $\phi $-equivariant unipotent isomorphism
\[
L_4 ^{\dR}(z_1 )\stackrel{\simeq }{\longrightarrow }L_4 ^{\dR}(z_2 )
\]
sends $T_{0,\infty }\in L_4 ^{\dR}(z_1 )$ to 
\[
\sum T_{0,\infty }+\sum a_{ij}T_{0,\infty }^i T_{1,\infty }^j [T_0 ,T_1 ]_{\infty }\in L_4 ^{\dR}(z_2 ),
\]
where
\begin{align*}
a_{00} & =-\int ^{z_2 }_{z_1 }(\omega _1 -t^{-2}dt) \\
a_{10} & =-\int ^{z_2 }_{z_1 }t^{-1}\omega _0 +\int ^{z_2 }_{z_1 }\omega _0 a_{00} \\
a_{01} & =\int ^{z_2 }_{z_1 }(\omega _1 -t^{-2}dt )a_{00} \\
a_{20} & =\int ^{z_2 }_{z_1 }\omega _0 a_{10} \\
a_{11} & =\int ^{z_2 }_{z_1 }\left(\frac{1}{2}(\omega _0 t^{-2}-t^{-2}dt)  +(\omega _1 -t^{-2}dt)a_{10}+\omega _0 a_{01}\right) \\
a_{02} & =\int ^{z_2 }_{z_1 }(\omega _1 -t^{-2}dt)a_{01}
\end{align*}

and sends $T_{1,\infty }\in L_4 (z_1 )$ to 
\[
\sum T_{1,\infty }+\sum b_{ij}T_{0,\infty }^i T_{1,\infty }^j [T_0 ,T_1 ]_{\infty }\in L_4 (z_2 ),
\]
where
\begin{align*}
b_{00} & =\int ^{z_2 }_{z_1 }\omega _0 \\
b_{10} & =\int ^{z_2 }_{z_1 }\omega _0 b_{00} \\
b_{01} & =\int ^{z_2 }_{z_1 }((\omega _1 -t^{-2}dt )b_{00}-t^{-1}\omega _0 ) \\
b_{20} & =\int ^{z_2 }_{z_1 }\omega _0 b_{10} \\
b_{11} & =\int ^{z_2 }_{z_1 }(\omega _0 b_{01}+\omega _1 b_{10}) \\
b_{02} & =\int ^{z_2 }_{z_1 }\left(\omega _1 b_{01}+\frac{t^{-2}}{2}(-\omega  +dt)\right).
\end{align*}
\end{lemma}
\begin{proof}
When $z_1 $ is congruent to $z_2 $, the formulas above describe the parallel transport isomorphism between the fibres of $L_4 ^{\dR}$ at $z_1 $ and $z_2 $. Hence the lemma follows from Besser's interpretation of iterated Coleman integration \cite{besser:tannakian}.
\end{proof}
\begin{lemma}
For $z_1 ,z_2 $ as above, if $z_2 $ is furthermore not a basepoint at infinity, the unique $\phi $-equivariant isomorphism
\[
L_4 ^{\dR}(z_1 )\stackrel{\simeq }{\longrightarrow }L_4 ^{\dR}(z_2 )
\]
is given by sending $T_{0,\infty }$ to $T_0 +\sum c_{ij}T_0 ^i T_1 ^j [T_0 ,T_1 ]$, where
\begin{align*}
c_{00}& =a_{00}+t(z_2 )^{-1} \\
c_{10} & = a_{10} \\
c_{01} & = a_{01}-\frac{1}{2}t(z_2 )^{-2} -t(z_2 )^{-1}a_{00} \\
c_{20} & = a_{20} \\
c_{11} & = a_{11}-\frac{t(z_2 )^{-1}}{2}-t(z_2 )^{-1}a_{10} \\
c_{02} & = a_{02}+\frac{t(z_2 )^{-3}}{6}-t(z_2 )^{-1}a_{01}+\frac{1}{2}t(z_2 )^{-2}a_{00},
\end{align*}
and sending $T_{1,\infty }$ to $T_1 +\sum d_{ij}T_0 ^i T_1 ^j [T_0 ,T_1 ]$, where
\begin{align*}
d_{00} & = b_{00} \\
d_{10} & = b_{10} \\
d_{01} & =b_{01}-t(z_2 )^{-1}b_{00} \\
d_{20} & = b_{20} \\
d_{11} & = b_{11} -t(z_2 )^{-1}b_{10} \\
d_{02} & = b_{02}-\frac{1}{2}t(z_2 )^{-1}-b_{01}t(z_2 )^{-1}+\frac{1}{2}t(z_2 )^{-2}b_{00} .
\end{align*}
\end{lemma}
\begin{proof}
We simply compose the isomorphism from the previous lemma with the coordinate transformation 
\[
A^{-1}= \left( \begin{array}{ccc}1 & 0 & 0 \\ 0 & 1 & 0 \\ -e^{t^{-1}T_1 }G_1 & -e^{t^{-1}T_1 }G_2 & e^{t^{-1}T_1 } \end{array} \right) .
\]
This gives 
\begin{align*}
\sum c_{ij}T_0 ^i T_1 ^j & =e^{t^{-1}T_1 }\cdot (-G_0 +\sum a_{ij}T_0 ^i T_1 ^j ).
\sum d_{ij}T_0 ^i T_1 ^j & =e^{t^{-1}T_1 }\cdot (-G_1 +\sum b_{ij}T_0 ^i T_1 ^j ).
\end{align*}
\end{proof}
We may rewrite this in a fashion similar to Lemma \ref{lemma:tedious} so that the formula in Lemma \ref{lemma:tedious} describes the $\phi $-equivariant isomorphism from $L_4 ^{\dR}(v)$ to $L_4 ^{\dR}(z)$ but the iterated integrals $a_*$ are suitably regularized.
We deduce the following.
\begin{lemma}
The unique unipotent $\phi $-equivariant splitting
\[
L_4 (z)\simeq V_{\dR}\oplus L_{1,4}(z)
\]
sends $T_0$ to 
\[
\left(T_0 ,-\sum c_{ij}T_0 ^i T_1 ^j [T_0 ,T_1 ]+e^{\int \omega _0 T_0 +(\int (\omega _1 -t^{-2}t^{-2}dt) -t^{-1}(z))T_1 }\left(\frac{1}{2}+\frac{1}{12}\log _p (\Delta )\right)T_0 [T_0 ,T_1 ]\right) 
\]
and sends $T_1 $ to 
\[
\left(T_1 ,-\sum d_{ij}T_0 ^i T_1 ^j [T_0 ,T_1 ]+e^{\int \omega _0 T_0 +(\int (\omega _1 -t^{-2}t^{-2}dt) -t^{-1}(z))T_1 }\left(\frac{1}{2}+\frac{1}{12}\log _p (\Delta )\right)T_1 [T_0 ,T_1 ]\right).
\]
\end{lemma}
\begin{proof}
The unique unipotent $\phi $-equivariant isomorphism
\[
L_4 (z)\to L_4 (v)
\]
sends $T_0$ to
\[
(T_0 ,e^{\int ^{z}_0 \omega _0 T_0 +\int ^{z}_0 (\omega _1 -t^{-2}dt)T_1 }\cdot \sum c_{ij}T_0 ^i T_1 ^j [T_0 ,T_1 ])
\]
and similarly for $T_1$. The unique unipotent $\phi $-equivariant isomorphism
\[
L_{1,4}(v)\to L_{1,4}(z)
\]
sends $T_{0,\infty }^i T_{1,\infty }^j [T_0 ,T_1 ]_{\infty }$ to
\[
e^{-\int ^{z}_0 \omega _0 T_0 -\int ^{z}_0 (\omega _1 -t^{-2}dt)\omega _1 T_1 }T_0 ^j T^j [T_0 ,T_1 ].
\]
Hence the lemma follows from the description of the unique $\phi $-equivariant splitting
\[
L_4 (v)\simeq V_{\dR}\oplus L_{1,4}(v)
\]
given in Lemma \ref{lemma:res_class}.
\end{proof}
Together with the description of the Hodge filtration in Lemma \ref{lemma:res_class}, we deduce the following.
\begin{proposition}\label{prop:main_class}
The class of $L_4 (z)$ in $\Hom (V_{\dR},L_{1,4}(z))/F^0 $ is given by
\begin{align*}
-\sum T_0 ^* \otimes c_{ij}T_0 ^i T_1 ^j [T_0 ,T_1 ]-\sum d_{ij}  T_1 ^* \otimes T_0 ^i T_1 ^i [T_0 ,T_1 ]+\frac{1}{2}\ad [T_0 ,T_1 ] \\
-e^{\int \omega _0 T_0 +(\int (\omega _1 -t^{-2}t^{-2}dt) -t^{-1}(z))T_1 }\left(\frac{1}{2}+\frac{1}{12}\log _p (\Delta )\right)\ad  [T_0 ,T_1 ].
\end{align*}
\end{proposition}
\begin{Remark}
Note that by construction, the class of $L_4 (z)$ in $\Ext ^1 (V,L_{1,4}(z))$ does not depend on a choice of auxiliary tangential basepoint $v$, unlike the class of $z$ in $H^1 _f (\Q _p ,U_3 (v))$. This tells us that the dependence of the $d_{ij}$ and $c_{ij}$ functions on a tangential basepoint is cancelled out by the dependence of $\log _p (\Delta )$. Recall that $v$ is taken to be a tangential basepoint dual to $\frac{dx}{2y}$ with respect to a Weierstrass model $y^2 =x^3 +ax+b$, and $\Delta $ is the discriminant of the cubic.
\end{Remark}
\section{Linearising the Selmer variety}\label{sec:linearizations}
In this section we study integral points on elliptic curves using the polylogarithm quotient $L_n (z)$. By Lemma \ref{logarithm}, using $L_n (z)$, we can associate to each point $z\in E(\Q )$ an extension class in  $\Ext^1 (V,\Sym^{n-2}(\kappa (z))(1))$, where $\kappa (z)$ is an extension of $\Q _p $ by $V$ whose extension class is the Kummer class of $z$ in $H^1 (G_{\Q },V)$. In this section we explain how to turn this into a map 
\begin{equation}\nonumber
 (\mathcal{E}-\{0 \} )(\mathbb{Z})\to \Ext^1 (V,E_{n}(V)(1))
\end{equation}
where $E_{n}(V)$ is a representation independent of $z$. We view this as a linearisation of the unipotent Albanese map.

\begin{definition}
Define $E(V)$ to be a continuous $\mathbb{Q}_p $-representation of $G_T$ which is an extension of 
$H^1 _f (G_T ,V)$ by $V$, with the property that for any map $c:\mathbb{Q}_p \to H^1 _f (G_T ,V)$, 
the extension $W$ obtained by pulling back $E(V)$ has class 
$c(1)$ in $H^1 _f (G_T ,V)$.
\end{definition}
The representation $E(V)$ has the property that for any crystalline extension 
\begin{equation}\nonumber
0\to V\to W \to \mathbb{Q}_p \to 0,
\end{equation}
there exists a unique Galois-equivariant morphism
\begin{equation}\nonumber
W\to E(V)
\end{equation}
giving a commutative diagram 
$$
\begin{tikzpicture}
\matrix (m) [matrix of math nodes, row sep=3em,
column sep=3em, text height=1.5ex, text depth=0.25ex]
{0 & V & W & \mathbb{Q}_p & 0  \\
0 & V & E(V) & H^1 _f (G_T ,V) & 0 \\};
\path[->]
(m-1-1) edge[auto] node[auto] { } (m-1-2)
(m-1-4) edge[auto] node[auto]{$[W]. \textbf{1}$} (m-2-4)
(m-1-2) edge[auto] node[auto]{$=$}(m-2-2)
(m-1-3) edge[auto] node[auto] { } (m-2-3)
(m-1-2) edge[auto] node[auto]  { } (m-1-3)
(m-1-3) edge[auto] node[auto] { } (m-1-4)
(m-1-4) edge[auto] node[auto] { } (m-1-5)
(m-2-1) edge[auto] node[auto] { } (m-2-2)
(m-2-2) edge[auto] node[auto] { } (m-2-3)
(m-2-3) edge[auto] node[auto] { } (m-2-4)
(m-2-4) edge[auto] node[auto] { }(m-2-5);
\end{tikzpicture}.$$

Hence for any extension $W$, we obtain a natural map of filtered Galois representations
\begin{equation}\nonumber
\Sym^k W \to \Sym^k E(V)
\end{equation}
which, restricted to the $i$th graded piece
\begin{equation}\nonumber
\Sym^i V \to \Sym^{k-i}H^1 _f (G_T ,V)\otimes \Sym^i (V),
\end{equation}
is given by
\begin{equation}\nonumber
v\mapsto [W]^{k-i}\otimes v.
\end{equation}
For any $c$ in $H^1 _f (G_T ,U_{n-1} )$, we obtain an extension of $V$ by $\Sym^{n-2}(E(V))(1)$.
\begin{definition}We let
 $E_n (V):=\Sym^{n-2}(E(V))$.
\end{definition}
This gives a natural map
\begin{equation}\nonumber
 \Psi _n :(\mathcal{E}-\{0 \} )(\mathbb{Z})\to \Ext ^1 (V, E_n (V)(1)).
\end{equation}
Since the target of $\Psi _n $ is now a vector space, this can be extended linearly to a new map, which we also denote by $\widetilde{\Psi } _n $,
\begin{equation}\nonumber
 \mathbb{Q}_p [H^1 _{f,T}(G_T ,U_{n-1} )]\to \Ext ^1 _{f,T}(V, E_n (V)(1)),
\end{equation}
where $T$ is a set of primes containing the set of primes of bad reduction. The short exact sequence 
\[
0\to V\to E (V)\to H^1 _f (G_{\Q },V)\to 0
\]
induces a decreasing filtration $(E_{n,i}(V))_i $ on $E_n (V)$ with $$E_{n,i}(V)/E_{n,i+1}(V)\simeq \Sym ^i (V)\otimes \Sym ^{n-2-i}H^1 _f (G_T ,V).$$ Let $B_{i,n}\subset \Q _p [H^1 _{f,T}(G_T ,U_{n-1} )]$ denote the kernel of the composite map
\[
\Q _p [H^1 _{f,T}(G_T ,U_{n-1} )]\stackrel{\widetilde{\Psi }_n }{\longrightarrow }H^1 _{f,T} (G_T ,V^* \otimes E_n (V)(1))\to H^1 _{f,T} (G_T ,V\otimes E_n (V)/E_{n-i}(V))).
\]
\begin{lemma}\label{Bin}
The image of $B_{i,n}$ in $H^1 _{f,T} (G_T ,V\otimes \Sym ^{i}(V))\otimes \Sym ^{n-2-i}H^1 _f (G_T ,V)$ is contained in $H^1 _{f,T} (G_T ,\Sym ^{i-1}(V)(1))\otimes \Sym ^{n-2-i}H^1 _f (G_T ,V)$.
\end{lemma}
\begin{proof}
This is a consequence of Lemma \ref{BCH} and Lemma \ref{wildeshaus}.
\end{proof}
\subsection{The crystalline universal extension}
We define the crystalline universal extension of $\Q _p $ by $V_{\dR}$ to be the extension of filtered $\phi $-modules 
\[
0\to V_{\dR}\to E_f (V_{\dR})\to \Q _p \to 0
\]
defined (as a vector space) by $E_f (V_{\dR}):=\Q _p \oplus V_{\dR}$, with 
\[
F^i E_f (V_{\dR}):=F^i \Q _p \oplus F^i V_{\dR}
\]
and 
\[
\phi _{E_f (V)}(\lambda ,v):=(\lambda ,\phi _V (v-\lambda T_0 )+\lambda T_0 )
\]
for $\lambda \in \Q _p $ and $v\in V_{\dR}$. Here we view $T_0 $ and $T_1 $ as a basis of $V_{\dR}$ dual to the basis $[\omega _0 ], [\omega _1 ]$ of $H^1 _{\dR}(E/\Q _p )$.

Given a point $z_1 \in E(\Z _p )$, we define a Frobenius structure on $\kappa ^{\dR}$ by the property that $z_1 ^* \kappa ^{\dR}$ is a split extension of $\Q _p $ by $V_{\dR}$ in the category of filtered $\phi $-modules. Denote this filtered $F$-isocrystal by $\kappa ^{\dR}_{z_1 }$.  We define $\kappa ^{\dR}(z_1 ,z_2 )$ to be the filtered $\phi $-module obtained by pulling back $\kappa ^{\dR}_{z_1 }$ along $z_2 $. This filtered $\phi $-module has basis $1,T_0 ,T_1 $, with $F^0 $ spanned by $1$ and $T_0$, and with a $\phi $-equivariant splitting of 
\[
0\to V_{\dR}\to \kappa ^{\dR}(z_1 ,z_2 )\to \Q _p \to 0
\]
given by 
\[
1\mapsto 1+\int ^{z_2 }_{z_1 }\omega _0 T_0 +\int ^{z_2 }_{z_1 }\omega _1 T_1 .
\]
\begin{lemma}\label{lemma:EfV}
There is a homomorphism of filtered $\phi $-modules
\[
i:\kappa ^{\dR} (z_1 ,z_2 )_{\dR}\to E_f (V_{\dR})
\]
given by 
\[
(\lambda ,v)\mapsto \left(\lambda \int ^{z_2 }_{z_1 }\omega _0 ,v+\lambda \int ^{z_2 }_{z_1 }\omega _0 T_0 +\int ^{z_2 }_{z_1 }\omega _1 T_1 \right).
\]
If $z_2 -z_1 $ is not torsion, this map is an isomorphism. There is a map of filtered $\phi $-modules
\[
\alpha :D_{\dR}(E(V))\to E_f (V_{\dR})
\]
giving commutative diagrams
\[
\begin{tikzcd}
D_{\dR}(\kappa (z_1 ,z_2 )) \arrow[d] \arrow[r] & \kappa ^{\dR}(z_1 ,z_2 ) \arrow[d] \\
D_{\dR}(E(V)) \arrow[r]           & E_f (V_{\dR})          
\end{tikzcd}
\]
where $\kappa (z_1 ,z_2 )$ denotes the extension of $\Q _p $ by $V$ corresponding to the point $z_2 -z_1 $ on $E$.
\end{lemma}
\begin{proof}
There are unipotent $\phi $-equivariant isomorphisms
\begin{align*}
s_1 &:\Q _p \oplus V_{\dR}\to E_f (V_{\dR}) \\
s_2 &:\Q _p \oplus V_{\dR}\to \kappa (z_1 ,z_2 ) \\
\end{align*}
sending $1$ to $(1,T_0  )$ and $(1,\int ^{z_2 }_{z_1 }\omega _0 T_0 +\int ^{z_2 }_{z_1 }\omega _1 T_1 )$ respectively.
\end{proof}
Taking symmetric powers, we deduce the following.
\begin{lemma}\label{lemma:iso_sym}
There is a map of filtered $\phi $-modules
\[
i_n :L_{1,n}^{\dR}(z)\to \Sym ^{n-2}(E_f (V_{\dR}))(1)
\]
given by sending $F(T_0 ,T_1 )[T_0 ,T_1 ]$ to
$$ \left(\int ^z _0 \omega _0 \right)^{n-2}e^{(\int ^z _0 (\omega _1 -t^{-2}dt)-t(z)^{-1})T_1 /\int ^z _0 \omega _0 }F\left(\frac{T_0 }{\int ^z _0 \omega _0 },\frac{T_1 }{\int ^z _0 \omega _0 }\right)\left(1-\frac{1}{2\int ^z _0 \omega _0 }T_0 T_1 \right)T_0 \wedge T_1 .$$

\end{lemma}
\begin{proof}
Recall that when $v$ is a tangential basepoint, we have an isomorphism of filtered $\phi $ -modules with an action of $V_{\dR}$
\[
\oplus _{i=0}^2 \Sym ^i (V_{\dR})(1) \simeq L_{1,4}(v)
\] 
sending $T_0 \wedge T_1 $ to $(1-\frac{1}{2}T_{0,\infty } T_{1,\infty } )[T_0 ,T_1 ]_{\infty }$.
\end{proof}
Putting everything together, we obtain the formula for the class of $$L_n ^{\dR}(z) \in \ext ^1 _{\fil ,\phi }(V_{\dR},\Sym ^{n-2}E_f (V_{\dR})(1))$$ which will be used to find equations for $X(\Z _p )_3$. Recall that $d_1 (z):=\int ^z _0 (\omega _1 -t^{-2}dt)-t(z)^{-1}$.
\begin{proposition}\label{prop:finally}
The image of $L_4 ^{\dR}(z)$ in $\ext ^1 _{\fil ,\phi }(V_{\dR},\Sym ^{2}E_f (V_{\dR})(1))$  under the isomorphisms 
\[
\ext ^1 _{\fil ,\phi }(V_{\dR},\Sym ^{n-2}E_f (V_{\dR})(1)) \simeq  \Hom (V_{\dR},\Sym ^{n-2}(E_f (V_{\dR}))(1))/F^0 
\] 
is given by

\begin{align*}
& -\left(\int ^z _0 \omega _0 \right)^2 \sum e^{-d_1 (z)T_1 /\int ^z _0 \omega _0 }\frac{c_{ij}}{(\int ^z _0 \omega _0 )^{i+j}}T_0 ^* \otimes T_0 ^i T_1 ^j [T_0 ,T_1 ]\\
& -\left(\int ^z _0 \omega _0 \right)^2 \sum e^{-d_1 (z)T_1 /\int ^z _0 \omega _0 }\sum \frac{d_{ij}}{(\int ^z _0 \omega _0 )^{i+j}}  T_1 ^* \otimes T_0 ^i T_1 ^i [T_0 ,T_1 ]+\frac{1}{2}\sum e^{-d_1 (z)T_1 /\int ^z _0 \omega _0 }\ad [T_0 ,T_1 ] \\
& -e^{T_0 }\int ^z _0 \omega _0 \left(\frac{1}{2}+\frac{1}{12}\log _p (\Delta )\right)\ad  [T_0 ,T_1 ].
\end{align*}
\end{proposition}
\begin{proof}
This follows from combining Lemma \ref{lemma:iso_sym} and Proposition \ref{prop:main_class}. 
\end{proof}
\begin{definition}\label{defn:fi}
Let $f_3 $ and $f_4 $ denote the $T_0 ^* \otimes T_0^2 [T_0 ,T_1 ]$ and $T_0 ^* \otimes T_0 T_1 [T_0 ,T_1 ]$ coefficients respectively. Then
\begin{align*}
f_3 (z) & = -c_{20}(z)-\left(\frac{1}{2}+\frac{1}{12}\log _p (\Delta )\right) \int ^z _0 \omega _0 \\
f_4 (z) & =-c_{11}(z)+d_1 (z)c_{10}(z)-\frac{1}{2}d_1 (z).
\end{align*}
\end{definition}
\subsection{The image of the Selmer variety under localisation}
We now describe how to use linearisation to compute the localisation map
\[
\Sel (U_n )\to H^1 _f (\Q _p ,U_n ).
\]
To describe this it will be helpful to assume the following conjecture.
\begin{Conjecture}\label{conj:BK}
For all $0\leq i\leq n$, the map 
 \begin{equation}\nonumber
\sum \loc_v  H^1 (G_T ,\Sym^i (V)(1))\to \oplus _{v\in T} H^1 (G_v ,\Sym^i (V)(1))
 \end{equation}
is injective, and $H^1 _f (G_T ,\Sym ^i (V)(-i))=0$.
\end{Conjecture}
This is clearly satisfied when $n=0$. When $n=1$, work of Kato \cite{kato:secret} reduces it to non-vanishing of a special value of a $p$-adic $L$-function: in each of the examples we work with, we check the non-vanishing of $L_p(E,s)$ at $s = 2$ by a computation in SageMath \cite{sage}. 
When this assumption is not satisfied the methods below may be modified, 
and will still give finiteness of the set of integral points provided $\sum \dim H^1 _{f,0}(G_T ,\gr _i U)<\sum \dim \gr _i U^{\dR}/F^0 $, where $\gr _i $ denotes the $i$th graded piece with respect to the central series filtration.

By standard results in the Galois cohomology of number fields \cite{milne:duality, washington1997galois}, the conjecture is implied by the Bloch--Kato conjectures. The conjecture implies (see \cite[Remark 1.2.4]{autours}) that $\dim H^1 _f (G_{\Q },\Sym ^i V(1))$ is equal to the 
dimension of the minus eigenspace (with respect to complex conjugation) of $\Sym^{i}(V^{\et} )(1)$, which is 
straightforwardly calculated to be $\lfloor i/2\rfloor $. 

\subsection{Equations}
Our strategy for finding equations in the rank greater than genus case can be described as follows. We factor the localisation map as 
\begin{align*}
\Sel (U_n )\to H^1 _f (G_{\Q _p },U_n )\times _{H^1 _f (G_{\Q _p },V)}J(\Q )\otimes \Q _p  \to H^1 _f (G_{\Q _p },U_n ).
\end{align*}
Concretely, if the logarithm map
\[
J(\Q )\otimes \Q _p \to H^1 _f (\Q _p ,V)
\]
is surjective, we may write $H^1 _f (G_{\Q _p },U_n )\times _{H^1 _f (G_{\Q _p },V)}J(\Q )\otimes \Q _p$ as $U_n ^{\dR}/F^0 \times \Q_p ^{r-g}$, by choosing generators $t_1 ,\ldots ,t_{r-g}$ for a complement $W$ to $\log _J ^*  H^0 (X,\Omega )$ inside $\Hom (J(\Q ), \Q _p )$. We then seek to find at least $r-g+1$ equations for the image of the Selmer scheme in $H^1 _f (G_{\Q _p },U_n )\times _{H^1 _f (G_{\Q _p },V)}J(\Q )\otimes \Q _p$. Given these equations, we may project onto $H^1 _f (G_{\Q _p },U_n )$ to obtain equations for $X(\Z _p )_{U_n} $. Note that, in the case $n=3$, $U_n$ is the maximal $n$-unipotent quotient of the $\Q _p $-\'etale fundamental group of $X_{\overline{\Q }}$, so $X(\Z _p )_{U_3 }=X(\Z _p )_3$.

We have a natural map
\[
H^1 _f (\Q _p ,U_{n-1})\times _{H^1 _f (\Q _p ,V)}H^1 _f (\Q ,V)\to H^1 _f (\Q _p ,V^* \otimes E_n (V)(1))
\]
sending a pair $(c_1 ,c_2 )$, where $c_1 \in H^1 _f (\Q _p ,U_n )$ and $c_2 \in H^1 _f (\Q ,V)$ have the same image $c_0$ in $H^1 _f (\Q _p ,V)$, to the image of $L_n (b)^{(c_1 )}$ in $H^1 _f (\Q _p ,V^* \otimes E_n (V)(1))$ via the map
\[
H^1 _f (\Q _p ,V^* \otimes \Sym ^{n-2}(\kappa (c_1 ))(1))\to H^1 _f (\Q _p ,V^* \otimes E_n (V)(1)) 
\]
defined by $c_2 $. 
Define $\widetilde{H}^1 _f (G_{\Q _p },V^* \otimes E_n (V)(1)) \subset $ to be the image of 
\[
\Q _p [H^1 _f (\Q _p ,U_{n-1})\times _{H^1 _f (\Q _p ,V)}H^1 _f (\Q ,V)]\to H^1 _f (\Q _p ,V^* \otimes E_n (V)(1)).
\]
Define $\widetilde{H}^1 _{f,T} (G_{\Q },V^* \otimes E_{n} (V)(1))\subset H^1 _{f,S}(G_{\Q },V^* \otimes E_n (V)(1))$ to be the image of $\Q _p [H^1 _{f,T} (G_{\Q},U_n )]$. 
\begin{lemma}\label{dim:bound}
\[
\dim \widetilde{H}^1 _{f,S} (G_{\Q} ,V^* \otimes E_4 (V)(1))\leq \binom{r+3}{3}+r\dim H^1 _{f,S}(G_{\Q },\Q _p (1))+\dim H^1 _{f,S}(G_{\Q },V(1)).
\]
\end{lemma}
\begin{proof}
This follows from Lemma \ref{Bin}.
\end{proof}
We have an obvious localisation map
\[
\widetilde{H}^1 _f (G_{\Q },V^* \otimes E_n (V)(1))\to \widetilde{H}^1 _f (G_{\Q _p },E_n (V)(1)),
\]

and a commutative diagram

\begin{center}
\begin{tikzcd}
H^1 _f (G_{\Q },U_n ) \arrow[r] \arrow[d] & \widetilde{H}^1 _f (G_{\Q },V^* \otimes E_n (V)(1)) \arrow[d] \\
H^1 _f (G_{\Q _p },U_n )\times _{H^1 _f (G_{\Q _p },V)}J(\Q )\otimes \Q _p  \arrow[r] & \widetilde{H}^1 _f (G_{\Q _p },V^* \otimes E_n (V)(1)).  \\
\end{tikzcd}
\end{center}

Let $e_n $ denote the dimension of $\widetilde{H}^1 _f (G_{\Q },V^* \otimes E_n (V)(1))$, and $f_n$ the dimension of $\widetilde{H}^1 _f (G_{\Q _p },E_n (V)(1))$. If $e_n <f_n$, then for any $m>e_n$ and any $m$-tuple $(z_i )$ of integral points we obtain the algebraic relation
\[
\rk ((\Psi _n (z_i ))_{i=1}^m )\leq e_n
\]
between the values of $\Psi _n $ on the $z_i$. Concretely, by choosing a basis of $\widetilde{H}^1 _f (G_{\Q _p },E_n (V)(1))$, this relation can be computed as the vanishing of all $(e_n +1)$-minors of the $m\times f_n $ matrix $(\Psi _n (z_i ))$. If we  are furthermore able to find an $e_n$-tuple of points $(z_i )$ for which the rank of $(\Psi _n (z_i ))$ is exactly equal to $e_n$, then we obtain the equation for $X(\Z )$.

Note that we have an isomorphism of filtered $\phi $-modules
\[
D_{\dR}(E(V))\simeq E_f (V)\oplus W,
\]
giving an isomorphism
\[
D_{\dR}(E_n (V))\simeq \oplus _{i=0}^{n-2} \Sym ^i (E_f (V_{\dR}))\otimes \Sym ^{n-2-i}(W).
\]
We obtain an isomorphism
\[
\ext ^1 _{\fil, \phi }(V,E_n (V)(1))\simeq \oplus _{i=0}^{n-2}\Sym ^{n-2-i}(W)\otimes (V_{\dR}\otimes \Sym ^i (V_{\dR}))/F^0 .
\]
The image of $L_n (z)$ under this isomorphism is equal to $\sum w^i \otimes [L_{n-i} ^{\dR}(z)]$, where $w\in W$ denotes the projection to $W$ of $\kappa (z)$.

Recall from the introduction that if the rank of $E$ is one, generated by $P\in E(\Q )$, we can recast the usual quadratic Chabauty formula as the vanishing of the function
\begin{equation}\label{eqn:det_ht}
z\mapsto \det \left( \begin{array}{cc} h_p (z)+\alpha & (\int ^z \omega )^2 \\ h(P) & (\int ^P \omega )^2 \\ \end{array} \right)
\end{equation}
on the set of integral points satisfying $\sum _{v\neq p}h_v (z)=\alpha $. 

Now suppose the rank of $E$ is two, and let 
\[
(f_1 ,\tau ):E(\Q )\otimes \Q _p \to \Q _p ^2
\]
be an isomorphism. For points $(z_1 ,\ldots ,z_n )$ in $E(\Q )$, let $H_1 (z_1 ,\ldots ,z_n )$ denote the matrix
\[
\left( \begin{array}{cccc} h (z_i ) & (\int ^{z_i } \omega )^2 & \tau (z_i )\int ^{z_i }\omega & \tau (z_i )^2 \end{array} \right) _{i=1,\ldots ,n}.
\]
Then the bilinearity of the $p$-adic height implies that, for any $n$, the rank of $H_1 (z_1 ,\ldots ,z_n )$ is at most three. 
To eliminate the variable $\tau $, and get an equation for $z\in X(\Z _p )$, we pass to depth three. Assuming $\rk H^1 _f (\Q ,T_p E(1))=1$ and that local heights of integral points at all primes away from $p$ are zero, we see from Lemma \ref{dim:bound} that $\widetilde{H}^1 _f (\Q ,V^* \otimes E_2 (V)(1))$ has dimension at most $5$, while $\widetilde{H}^1 _f (\Q _p ,V^* \otimes E_2 (V)(1))$ has dimension 8. This means that, if $f_3$ and $f_4 $ are the functions from Definition \ref{defn:fi}, 
then for any $n$-tuple of points $(z_i )$ where all the $z_i$ are integral, the rank of the matrix $H_2 (z_1 ,\ldots ,z_n )$ given by
\[
\left( \begin{array}{cccccc} (\int ^{z_i }\omega )^3 & \tau (z_i )(\int ^{z_i }\omega )^2 & \tau (z_i )^2 \int ^{z_i }\omega  & \tau (z_i )^3 & f_3 (z_i) & f_4 (z_i ) \end{array} \right) _{i=1,\ldots ,n}
\]
is at most 5. In particular, if we can find integral points $(P_i )_{i=1,\ldots ,5}$ such that $H_1 ((P_i )_{i=1}^3 )$ has rank three and $H_2 ((P_i )_{i=1}^5 )$ has rank 5, then for any $z\in E(\Z )$, we have
\[
\det (H_1 (P_1 ,P_2 ,P_3 ,z))=\det (H_2 (P_1 ,\ldots ,P_5 ,z))=0.
\]
If we now view $\tau (z)$ as a formal variable $t$, these are two functions on $X(\Z _p )$ with values in $\Q _p [t]$, and taking resultants we obtain
\[
\Res (\det (H_1 (P_1 ,P_2 ,P_3 ,z)),\det (H_2 (P_1 ,\ldots ,P_5 ,z)))=0.
\]

If $\rk H^1 _f (\Q ,T_p E(1))=1$, all local Tamagawa numbers are equal to 1, and the local heights of integral points are not all equal to $1$, then the rank of the matrix $H_2 (x_1 ,\ldots ,x_n )$ given by 
\[
\left( \begin{array}{cccccccc} (\int ^{z_i }\omega )^3 & \tau (z_i )(\int ^{z_i }\omega )^2 & \tau (z_i )^2 \int ^{z_i }\omega  & \tau (z_i )^3 & \int ^{z_i }\omega & \tau (z_i ) & f_3 (z_i) & f_4 (z_i ) \end{array} \right) _{i=1,\ldots ,n}
\]
is at most 7. Then, similarly, if one can find integral points $P_i $ such that $\rk H_1 (P_1 ,P_2 ,P_3 )=3$ and $\rk H_2(P_1 ,\ldots ,P_7 )=7$, then the function
\[
\Res (\det (H_1 (P_1 ,P_2 ,P_3 ,z)),\det (H_2 (P_1 ,\ldots ,P_7 ,z)))=0
\]
contains $E(\Z )$, and has finitely many zeroes on $X(\Z _p )$. 

If we relax the local conditions, these formulas become a little more complicated. Fix a set $\gamma =(\gamma _v )\in \prod _{v\in T_0 }H^1 (G_v ,U_3 )$ of local conditions. Recall that we have an isomorphism
\[
H^1 (G_v ,U_3 )\simeq H^1 (G_v ,\Q _p (1))\times H^1 (G_v ,V(1)),
\]
and that $H^1 (G_v ,V(1))$ is one dimensional if $v$ is a prime of split multiplicative reduction, and zero otherwise. If the images of $\gamma _v $ in $H^1 (G_v ,V(1))$ are zero for all primes $v$, then we may define $H_2 $ to be the same matrix as above, where the $P_i$ are now all required to lie in $X(\Z )_{\gamma }$. If some of the $\gamma _v $ have nonzero image in $H^1 (G_v ,V(1))$, we instead take 
\[
H_2 (P_i )=\left( \begin{array}{ccccccccc} (\int ^{z_i }\omega )^3 & \tau (z_i )(\int ^{z_i }\omega )^2 & \tau (z_i )^2 \int ^{z_i }\omega  & \tau (z_i )^3 & \int ^{z_i }\omega & \tau (z_i ) & f_3 (z_i) &  f_4 (z_i ) & 1 \end{array} \right) _{i=1,\ldots ,n} 
\]
(recall that the image of $\gamma _v$ in $H^1 (G_v ,V(1))$ is described explicitly in terms of its image in the special fibre of the N\'eron model at $v$, via Theorem \ref{thm:SS}). 
Again, if one can find $P_1 ,\ldots ,P_8 \in X(\Z )_{\gamma }$ such that $H_2 (P_1 ,\ldots ,P_8 )$ has rank 8, then one obtains a function
\[
\Res (\det (H_1 (P_1 ,P_2 ,P_3 ,z)),\det (H_2 (P_1 ,\ldots ,P_7 ,z)))=0
\]
vanishing on $X(\Z )_{\gamma}$.
\begin{lemma}\label{lemma:main}
Assume $\rk E(\Q )=2$ and Conjecture \ref{conj:BK} holds. With $H_1 $ and $H_2 $ as above, 
\[
X(\Z _p )_{3,\gamma }\subset \{ \Res (\det (H_1 (P_1 ,P_2 ,P_3 ,z)),\det (H_2 (P_1 ,\ldots ,P_7 ,z)))=0\} \subset X(\Z _p ).
\]
\end{lemma}
\begin{Remark}
The assumption that all the $P_i$ are in $X(\Z )_{\gamma }$ is unnecessarily restrictive. One can apply a similar determinantal construction to obtain functions vanishing on $X(\Z _p )_{3,\gamma }$ using $P_i $ which are merely in $E(\Q )$, however the conditions become more difficult to state.
\end{Remark}

\subsection{Relation to the work of Goncharov and Levin}\label{S:GL} 
Our determinantal relations for integral points may be rephrased in the following way.
\begin{lemma}
Suppose $c _i \in \Q _p $, $P_1 ,\ldots ,P_n \in E(\Q )$ have the following properties.
\begin{enumerate}
\item $\sum c _i [P_i ] ^3 =0$ in $\Sym ^3 E(\Q )\otimes \Q _p $.
\item For all primes $v$, $\sum c_i \lambda _v (P_i )\otimes P_i =0 $ in $E(\Q )\otimes \Q _p $.
\item For all primes $v$ of split multiplicative reduction, $P_i$ is $v$-integral and
\[
\sum c_i B_3 (m_v (P_i )/N_v )=0
\]
in $\Q _p $, where $m_v $ and $N_v$ are as in Theorem \ref{thm:SS}.
\end{enumerate}
Then the image of $\sum c_i [P_i ]$ in $H^1 (\Q ,V^* \otimes E_4 (V)(1))$ lies in the image of $H^1 _f (\Q ,V(1))$ under the map induced by
\[
V(1)\hookrightarrow V^* \otimes \Sym ^2 (V)(1)\hookrightarrow V^* \otimes E_4 (V)(1).
\]
\end{lemma}
Indeed, if $\dim H^1 _f (\Q ,V(1))<\dim H^1 _f (\Q _p ,V(1))$, then this gives the determinantal identities satisfied by $P_i$ above. The analogous conditions, with values in real vector spaces rather than $p$-adic vector spaces, are exactly the ones used by Goncharov and Levin \cite{goncharov-levin} in their construction of divisors on elliptic curves for which the value of the (real) elliptic dilogarithm is related to $L(E,2)$:

\begin{theorem}[Goncharov--Levin]
Let $E$ be a modular elliptic curve over $\Q$. Then there exists a $\Q$-rational divisor $D = \sum n_j(P_j)$ over $\overline{\Q}$ such that 
\begin{equation*}
L(E,2) = q(D) \pi D^{E}(D)
\end{equation*}
for some $q(D)\in \Q^{\times}$.
Furthermore, $D$ satisfies the following properties:
\begin{enumerate}
\item \label{it:sym} $\sum n_j P_j\otimes P_j \otimes P_j = 0 \in \sym^3 E(\overline{\Q})$.
\item \label{it:local_height} For any valuation $v$ of $\Q(D)$ we have
\begin{equation*}
\sum_{j}n_j \hat{\lambda}_v(P_j)P_j = 0  \quad \text{in} \quad  E(\overline{\Q})\otimes \R 
\end{equation*}
where $\hat{\lambda}_v$ is the model-independent canonical local $v$-adic height.
\item \label{it:integrality} There is an integrality condition at the primes of split multiplicative reduction.
\end{enumerate}
\end{theorem}

\subsection{Characterisation of the functions $f_3 $ and $f_4 $}
If $\omega $ is an element of $K(\! (t)\! )[\log (t)]\cdot dt$, lying in the image of $K(\! (t)\! )[\log (t)]$ under the map
\[
d: K(\! (t)\! )[\log (t)]\to K(\! (t)\! )[\log (t)]dt,
\]
then we write $\int \omega $ to be the unique preimage under $d$ whose constant term is zero. We similarly define a formal iterated integral
\[
\int \omega _0 \ldots \omega _n 
\]
inductively, if for all $i$, $\omega _i \int \omega _{i+1 }\ldots \omega _n \in K(\! (t)\! )[\log (t)]$ is in the image of $d$. When we want to emphasise the parameter, we write this as $\int ^z \omega _0 \ldots \omega _n $.
\begin{enumerate}
\item $d_1$ is the unique Coleman function satisfying $dd_1 =\omega _1$ which is odd with respect to the involution $[-1]$.
\item $a_{10}$ is the unique Coleman function whose expansion at $0$ is given by the formal iterated integral
\[
\int \omega _0 \omega _1 .
\]

\item $f_3$ is the unique Coleman function which vanishes at $0$ and satisfies 
\[
df_3=-\omega _0 a_{10}-\left(\frac{1}{2}+\frac{1}{12}\log _p (\Delta )\right)\omega _0 .
\]
In the terminology above, $f_3 $ is a Coleman function whose formal power series expansion at $0$ is equal to
\[
f_3 =-\int \omega _0 \omega _0 \omega _1-\left(\frac{1}{2}+\frac{1}{12}\log _p (\Delta )\right)\int \omega _0 .
\]
\item $f_4$ is the unique holomorphic Coleman function on $E$, which is zero at $0$, with vanishing first derivative, whose formal power series expansion at $0$ equals
\[
\int \omega _0 \omega _1 \omega _1 -\frac{1}{2}\int \omega _1 .
\]
\end{enumerate}
Using this, we can relate these functions to the $p$-adic sigma function. 

In the following two sections, we describe how to compute the set $X(\Z_p)_{3,\gamma}$ in the case of $\rk E(\Q) = 2$ as well as the analogous sets when $\rk E(\Q) \leq 1$ and give numerical examples.
\section{Computing triple integrals}\label{sec:triple}
The strategy of computing triple Coleman integrals follows the approach in \cite{balakrishnan:iterated}: for a curve $X$ and points $P, Q \in X(\Q_p)$, start by fixing  a basis $\{\xi_i\}$ of $H^1_{\dR}(X)$, and pull back along a lift of $p$-power Frobenius $\phi$ to compute the integral \begin{equation}\label{ref:expand}\int_{\phi(P)}^{\phi(Q)} \xi_i \xi_k \xi_l  = \int_P^Q \phi^*(\xi_i) \phi^*(\xi_k)\phi^*(\xi_l).\end{equation} Then expand the right hand side of \eqref{ref:expand} using Kedlaya's algorithm \cite{kedlaya:mw} (for a hyperelliptic curve) or Tuitman's algorithm \cite{tuitman:p1, tuitman:pc-general} (for a general curve), which gives \begin{equation}\label{frob}\phi^*(\xi_{\ell}) = df_{\ell} + \sum_m M_{\ell m} \xi_m.\end{equation} Substituting the expression of \eqref{frob} into the right hand side of \eqref{ref:expand} and expanding, one can then express it as a sum of double Coleman integrals. We carry this out simultaneously for all triples of basis elements $\{\xi_i\}$ and adjust endpoints from $P$ to $\phi(P)$ and $Q$ to $\phi(Q)$ using the following analogue of additivity in endpoints:

\begin{lemma}\label{link}
We have 
\begin{align*}\int_P^Q \xi_i\xi_j\xi_k &= \int_P^{P'} \xi_i\xi_j\xi_k + \int_{P'}^{Q'}\xi_i\int_P^{P'}\xi_j\xi_k + \int_{P'}^{Q'}\xi_i\xi_j\int_P^{P'}\xi_k + \int_{P'}^{Q'}\xi_i\xi_j\xi_k\\
 &\qquad +  \int_{Q'}^Q \xi_i\int_P^{P'}\xi_j\xi_k + \int_{Q'}^Q \xi_i\int_{P'}^{Q'}\xi_j\int_{P}^{P'}\xi_k + \int_{Q'}^Q\xi_i\int_{P'}^{Q'}\xi_j\xi_k \\
 &\qquad + \int_{Q'}^Q\xi_i\xi_j\int_{P}^{P'}\xi_k + \int_{Q'}^Q\xi_i\xi_j\int_{P'}^{Q'}\xi_k + \int_{Q'}^{Q}\xi_i\xi_j\xi_k.\end{align*}
\end{lemma}

We also recall the shuffle product identity for 1-forms $\eta_i$, which we will use to simplify certain terms in the linear system:
\begin{equation}\label{eq:shuffle}
\left(\int_P^Q \eta_1\right)  \left(\int_P^Q \eta_2\eta_3\right) = \int_P^Q \eta_1\eta_2\eta_3 + \int_P^Q \eta_2\eta_1\eta_3 + \int_P^Q \eta_2\eta_3\eta_1.
\end{equation}
Now we give an explicit formulation of the fundamental linear system for triple integrals.  We begin by expanding \eqref{ref:expand}, which yields the following:

\begin{align}\label{biglin}
\int_{\phi(P)}^{\phi(Q)} \xi_i \xi_k \xi_l
  &= \int_P^Q \phi^*(\xi_i)\phi^*(\xi_k)\phi^*(\xi_l) \nonumber\\
  &= \int_P^Q
      \left(df_i + \sum_{a=0}^{2g-1} M_{ia}\xi_a\right)
      \left(df_k + \sum_{b=0}^{2g-1} M_{kb}\xi_b\right)
      \left(df_l + \sum_{c=0}^{2g-1} M_{lc}\xi_c\right) \nonumber\\
  &= \int_P^Q df_i df_k df_l
   + \left(\sum_{a = 0}^{2g-1} M_{ia}\xi_a\right) df_k df_l
   + df_i \left(\sum_{b=0}^{2g-1} M_{kb}\xi_b\right) df_l\\
  &\qquad
   + \left(\sum_{a=0}^{2g-1} M_{ia}\xi_a\right)
     \left(\sum_{b=0}^{2g-1} M_{kb}\xi_b\right) df_l
   + df_i df_k \left(\sum_{c=0}^{2g-1} M_{lc}\xi_c\right) \nonumber\\
  &\qquad
   + \left(\sum_{a=0}^{2g-1} M_{ia}\xi_a\right)
     df_k
     \left(\sum_{c=0}^{2g-1} M_{lc}\xi_c\right)
   + df_i
     \left(\sum_{b=0}^{2g-1} M_{kb}\xi_b\right)
     \left(\sum_{c=0}^{2g-1} M_{lc}\xi_c\right) \nonumber\\
  &\qquad
   + \left(\sum_{a=0}^{2g-1} M_{ia}\xi_a\right)
     \left(\sum_{b=0}^{2g-1} M_{kb}\xi_b\right)
     \left(\sum_{c=0}^{2g-1} M_{lc}\xi_c\right) \nonumber\\
  &= c_{ikl}
   + \int_P^Q
     \left(\sum_{a=0}^{2g-1} M_{ia}\xi_a\right)
     \left(\sum_{b=0}^{2g-1} M_{kb}\xi_b\right)
     \left(\sum_{c=0}^{2g-1} M_{lc}\xi_c\right)\nonumber\\
    &= c_{ikl}
  + \sum_{a,b,c=0}^{2g-1} M_{ia}M_{kb}M_{lc}
    \int_P^Q \xi_a\xi_b\xi_c. \nonumber
\end{align}

We now describe the computation of the constants $c_{ikl}$. By definition, $c_{ikl}$ is the sum of the first seven terms on the right-hand side of \eqref{biglin}; each of these terms can be evaluated in terms of single and double Coleman integrals. For convenience, write
\begin{align*}
c_{ikl} = T_1 + T_2 + T_3 + T_4 + T_5 + T_6 + T_7,
\end{align*}
where
\begin{align*}
T_1 &= \int_P^Q df_i\,df_k\,df_l,\\
T_2 &= \int_P^Q \left(\sum_{a=0}^{2g-1} M_{ia}\xi_a\right) df_k\,df_l,\\
T_3 &= \int_P^Q df_i \left(\sum_{b=0}^{2g-1} M_{kb}\xi_b\right) df_l,\\
T_4 &= \int_P^Q \left(\sum_{a=0}^{2g-1} M_{ia}\xi_a\right)\left(\sum_{b=0}^{2g-1} M_{kb}\xi_b\right) df_l,\\
T_5 &= \int_P^Q df_i\,df_k \left(\sum_{c=0}^{2g-1} M_{lc}\xi_c\right),\\
T_6 &= \int_P^Q \left(\sum_{a=0}^{2g-1} M_{ia}\xi_a\right) df_k \left(\sum_{c=0}^{2g-1} M_{lc}\xi_c\right),\\
T_7 &= \int_P^Q df_i \left(\sum_{b=0}^{2g-1} M_{kb}\xi_b\right)\left(\sum_{c=0}^{2g-1} M_{lc}\xi_c\right).
\end{align*}
We evaluate these terms in turn.  First, for $T_1$ we have \small
 \begin{align}\int_P^Q df_idf_kdf_l &=\int_P^Q df_i(R_1)\int_P^{R_1}df_k(R_2)\int_P^{R_2}df_l\\  \nonumber
&=\int_P^Q df_i(R_1)\int_P^{R_1}df_k(R_2)(f_l(R_2) - f_l(P))\\ \nonumber
&=\int_P^Q df_i(R_1)\int_P^{R_1}df_kf_l - f_l(P)\int_P^Qdf_i(R_1)\int_P^{R_1}df_k\\ \nonumber
&= f_i(Q)\int_P^Q df_k f_l - \int_P^Q f_idf_kf_l - f_l(P)\left(\int_P^Q df_if_k - f_k(P)(f_i(Q)-f_i(P)) \right). \nonumber
\end{align}
\normalsize

Next, for $T_2$ we compute

\tiny
\begin{align}
\int_P^Q\left(\sum_{a=0}^{2g-1} M_{ia}\xi_a\right)df_kdf_l
&= \int_P^Q \left(\sum_{a=0}^{2g-1} M_{ia}\xi_a\right)(R_1)
   \int_P^{R_1} df_k(R_2)\int_P^{R_2} df_l \nonumber\\
&= \int_P^Q \left(\sum_{a=0}^{2g-1} M_{ia}\xi_a\right)(R_1)
   \int_P^{R_1} df_k(R_2)\bigl(f_l(R_2)-f_l(P)\bigr) \nonumber\\
&= \int_P^Q \left(\sum_{a=0}^{2g-1} M_{ia}\xi_a\right)(R_1)
   \int_P^{R_1} df_k\,f_l
   - f_l(P)\int_P^Q \left(\sum_{a=0}^{2g-1} M_{ia}\xi_a\right)(R_1)
   \int_P^{R_1} df_k \nonumber\\
&= \int_P^Q \left(\sum_{a=0}^{2g-1} M_{ia}\xi_a\right)(R_1)
   \int_P^{R_1} df_k\,f_l
   - f_l(P)\int_P^Q \left(\sum_{a=0}^{2g-1} M_{ia}\xi_a\right)
   \bigl(f_k(R_1)-f_k(P)\bigr) \nonumber\\
&= \int_P^Q \left(\sum_{a=0}^{2g-1} M_{ia}\xi_a\right)(R_1)
   \int_P^{R_1} df_k\,f_l
   - f_l(P)\int_P^Q \sum_{a=0}^{2g-1} M_{ia}\xi_a\,f_k
   + f_l(P)f_k(P)\int_P^Q \sum_{a=0}^{2g-1} M_{ia}\xi_a .
\end{align}
\normalsize

For $T_3$ we similarly obtain
\tiny
\begin{align}
\int_P^Q df_i\left(\sum_{b=0}^{2g-1} M_{kb}\xi_b\right)df_l
&= \int_P^Q df_i(R_1)
   \int_P^{R_1}\left(\sum_{b=0}^{2g-1} M_{kb}\xi_b\right)(R_2)
   \int_P^{R_2} df_l \nonumber\\
&= \int_P^Q df_i(R_1)
   \int_P^{R_1}\left(\sum_{b=0}^{2g-1} M_{kb}\xi_b\right)(R_2)
   \bigl(f_l(R_2)-f_l(P)\bigr) \nonumber\\
&= \int_P^Q df_i(R_1)
   \int_P^{R_1}\left(\sum_{b=0}^{2g-1} M_{kb}\xi_b\right)f_l
   - f_l(P)\int_P^Q df_i(R_1)
   \int_P^{R_1}\left(\sum_{b=0}^{2g-1} M_{kb}\xi_b\right) \nonumber\\
&= f_i(Q)\int_P^Q \sum_{b=0}^{2g-1} M_{kb}\xi_b\,f_l
   - \int_P^Q f_i\sum_{b=0}^{2g-1} M_{kb}\xi_b\,f_l \nonumber\\
&\qquad
   - f_l(P)\left(
      f_i(Q)\int_P^Q \sum_{b=0}^{2g-1} M_{kb}\xi_b
      - \int_P^Q f_i\sum_{b=0}^{2g-1} M_{kb}\xi_b
   \right).
\end{align}
\normalsize

For $T_4$ we have\tiny
\begin{align}
\int_P^Q
\left(\sum_{a=0}^{2g-1} M_{ia}\xi_a\right)
\left(\sum_{b=0}^{2g-1} M_{kb}\xi_b\right)df_l
&= \int_P^Q \left(\sum_{a=0}^{2g-1} M_{ia}\xi_a\right)(R_1)
   \int_P^{R_1}\left(\sum_{b=0}^{2g-1} M_{kb}\xi_b\right)(R_2)
   \int_P^{R_2} df_l \nonumber\\
&= \int_P^Q \left(\sum_{a=0}^{2g-1} M_{ia}\xi_a\right)(R_1)
   \int_P^{R_1}\left(\sum_{b=0}^{2g-1} M_{kb}\xi_b\right)(R_2)
   \bigl(f_l(R_2)-f_l(P)\bigr) \nonumber\\
&= \int_P^Q \left(\sum_{a=0}^{2g-1} M_{ia}\xi_a\right)(R_1)
   \int_P^{R_1}\left(\sum_{b=0}^{2g-1} M_{kb}\xi_b\right)f_l \nonumber\\
&\qquad
   - f_l(P)\int_P^Q
   \left(\sum_{a=0}^{2g-1} M_{ia}\xi_a\right)
   \left(\sum_{b=0}^{2g-1} M_{kb}\xi_b\right).
\end{align}
\normalsize

For $T_5$ we compute\tiny
\begin{align*}
\int_P^Q df_i df_k \left(\sum_{c=0}^{2g-1} M_{lc}\xi_c\right)
&= \int_P^Q df_i(R)\int_P^R df_k \left(\sum_{c=0}^{2g-1} M_{lc}\xi_c\right) \\
&= \left(
      f_i(R)\int_P^R df_k \left(\sum_{c=0}^{2g-1} M_{lc}\xi_c\right)
   \right)_{R=P}^{R=Q}
   - \int_P^Q f_i\,df_k \left(\sum_{c=0}^{2g-1} M_{lc}\xi_c\right) \\
&= f_i(Q)\int_P^Q df_k \left(\sum_{c=0}^{2g-1} M_{lc}\xi_c\right)
   - \int_P^Q (f_i\,df_k)\left(\sum_{c=0}^{2g-1} M_{lc}\xi_c\right).
\end{align*}
\normalsize

We postpone the computation of $T_6$ momentarily and first write $T_7$ in the form\tiny
\begin{align*}
\int_P^Q df_i
\left(\sum_{b=0}^{2g-1} M_{kb}\xi_b\right)
\left(\sum_{c=0}^{2g-1} M_{lc}\xi_c\right)
&= \int_P^Q df_i(R)
   \int_P^R
   \left(\sum_{b=0}^{2g-1} M_{kb}\xi_b\right)
   \left(\sum_{c=0}^{2g-1} M_{lc}\xi_c\right) \\
&= \left(
      f_i(R)\int_P^R
      \left(\sum_{b=0}^{2g-1} M_{kb}\xi_b\right)
      \left(\sum_{c=0}^{2g-1} M_{lc}\xi_c\right)
   \right)_{R=P}^{R=Q} \\
&\qquad
   - \int_P^Q
   f_i\left(\sum_{b=0}^{2g-1} M_{kb}\xi_b\right)
   \left(\sum_{c=0}^{2g-1} M_{lc}\xi_c\right) \\
&= f_i(Q)\int_P^Q
   \left(\sum_{b=0}^{2g-1} M_{kb}\xi_b\right)
   \left(\sum_{c=0}^{2g-1} M_{lc}\xi_c\right) \\
&\qquad
   - \int_P^Q
   f_i\left(\sum_{b=0}^{2g-1} M_{kb}\xi_b\right)
   \left(\sum_{c=0}^{2g-1} M_{lc}\xi_c\right).
\end{align*}
\normalsize

Finally, we use the shuffle product identity \eqref{eq:shuffle}, together with a variant of the computation of $T_7$, to simplify $T_6$:

\tiny
\begin{align*}
\int_P^Q
\left(\sum_{a=0}^{2g-1} M_{ia}\xi_a\right)
df_k
\left(\sum_{c=0}^{2g-1} M_{lc}\xi_c\right)
&=
\left(\int_P^Q \sum_{a=0}^{2g-1} M_{ia}\xi_a\right)
\left(\int_P^Q df_k \sum_{c=0}^{2g-1} M_{lc}\xi_c\right) \\
&\qquad
 - \int_P^Q df_k
   \left(\sum_{a=0}^{2g-1} M_{ia}\xi_a\right)
   \left(\sum_{c=0}^{2g-1} M_{lc}\xi_c\right) \\
&\qquad
 - \int_P^Q df_k
   \left(\sum_{c=0}^{2g-1} M_{lc}\xi_c\right)
   \left(\sum_{a=0}^{2g-1} M_{ia}\xi_a\right) \\
&=
\left(\int_P^Q \sum_{a=0}^{2g-1} M_{ia}\xi_a\right)
\left(
   f_k(Q)\int_P^Q \sum_{c=0}^{2g-1} M_{lc}\xi_c
   - \int_P^Q f_k \sum_{c=0}^{2g-1} M_{lc}\xi_c
\right) \\
&\qquad
-\left(
   f_k(Q)\int_P^Q
   \left(\sum_{a=0}^{2g-1} M_{ia}\xi_a\right)
   \left(\sum_{c=0}^{2g-1} M_{lc}\xi_c\right)
   - \int_P^Q
     f_k\left(\sum_{a=0}^{2g-1} M_{ia}\xi_a\right)
     \left(\sum_{c=0}^{2g-1} M_{lc}\xi_c\right)
 \right) \\
&\qquad
-\left(
   f_k(Q)\int_P^Q
   \left(\sum_{c=0}^{2g-1} M_{lc}\xi_c\right)
   \left(\sum_{a=0}^{2g-1} M_{ia}\xi_a\right)
   - \int_P^Q
     f_k\left(\sum_{c=0}^{2g-1} M_{lc}\xi_c\right)
     \left(\sum_{a=0}^{2g-1} M_{ia}\xi_a\right)
 \right).
\end{align*}
\normalsize

Summing the resulting expressions for $T_1,\dots,T_7$ yields the constant $c_{ikl}$. We then have \begin{equation}\label{needscorr}\int_{\phi(P)}^{\phi(Q)} \xi_i \xi_k \xi_l
= c_{ikl} + \sum_{a,b,c=0}^{2g-1} M_{ia}M_{kb}M_{lc}\int_P^Q \xi_a\xi_b\xi_c.\end{equation} Finally, correcting endpoints on the left-hand side of \eqref{needscorr} using Lemma \ref{link} and noting that 1 is not an eigenvalue of $(I - M^{\otimes 3})$, one solves the linear system whose values recover the triple integrals on all basis elements of de Rham cohomology. We give an implementation in \cite{git:repo}, building on prior work of \cite{bt}.

\subsection{Lemmas for elliptic curves}
In what follows, we will compute identities between formal triple integrals in $K(\! (t)\! )[\log (t)]$ in the sense above. These will be used to give algorithms for computing the Coleman functions $f_3$ and $f_4$ defined earlier.

 Let $\omega_0 = \frac{dx}{2y}$ and $\omega_1 = x \frac{dx}{2y}$ on an elliptic curve $E$ in the form $y^2 = f(x)$.  First, the shuffle product identity immediately gives
\begin{align*}
\int_P^Q \a\a\a &= \frac{1}{6}\left(\int_P^Q \a\right)^3\\
\int_P^Q \b\b\b &= \frac{1}{6}\left(\int_P^Q \b\right)^3
\end{align*}

and 
\begin{align}
2\int_P^Q \a\a\b + \int_P^Q\a\b\a &= \int_P^Q \a \int_P^Q \a\b \label{eq7}\\
2\int_P^Q \b\b\a + \int_P^Q \b\a\b &= \int_P^Q \b \int_P^Q \b\a \nonumber \\
2\int_P^Q \a\b\b + \int_P^Q \b\a\b &= \int_P^Q\b \int_P^Q \a\b \label{eq10}\\
2\int_P^Q\b\a\a + \int_P^Q \a\b\a &= \int_P^Q \a \int_P^Q \b\a \nonumber \end{align}

\begin{lemma}\label{lemma17}  
We have equalities of formal power series in $K(\! (t)\! )[\log (t)]$
\begin{align}\int ^z \a\b\a &= \frac{1}{2}\int_{-z}^z \a\b\a - \int  ^z \a \int ^z \a\b \label{eq5} \\
\int^z \b\a\b &= \frac{1}{2}\int_{-z}^z \b\a\b +  \int ^z \b \int ^z \a\b - \left(\int  ^z \a\right)^2\int ^z \b \label{eq6}  \\
\int ^z \a\a\b &= \frac{1}{2}\left(\int  ^z\a \int ^z \a \b -\int ^z \a\b\a\right) \label{eq8}\\
\int ^z \a\b\b &=\frac{1}{2}\left(\int ^z  \b \int ^z \a \b  - \int ^z \b\a\b\right) \label{eq9}.\end{align}
\end{lemma}

\begin{proof}We begin by deriving \eqref{eq5}. Using Lemma \ref{link}, we have \begin{equation}\label{linkforbP}\int_{-z}^z \a\b\a = \int_{-z}  \a\b\a + \int ^z \a \int_{-z}  \b\a + \int ^z \a\b\int_{-z}  \a + \int ^z \a\b\a.\end{equation}
By applying the negation map to each set of endpoints and integrands below, as well as using Prop 2.1(3) of \cite{balakrishnan:iterated}, we have the following identities for single, double, and triple integrals:
\begin{align*}\int_{-z}  \a &= \int_z (-1)\a = \int ^z (-1)\a = \int ^z \a,\\
\int_{-z} \b\a &= \int_z  (-1)^2 \b\a = \int ^z (-1)^2(-1)^2 \a\b = \int ^z \a\b, \\
\int_{-z} \a\b\a &= \int_{z} (-1)^3\a\b\a = \int ^z (-1)^3(-1)^3 \a\b\a = \int ^z \a\b\a \end{align*}

Substituting these quantities into \eqref{linkforbP} yields $$\int ^z \a\b\a = \frac{1}{2}\int_{-z}^z \a\b\a - \int ^z \a \int ^z \a\b,$$ as desired.

Now swapping $\a$ and $\b$ and using the analogue of Lemma \ref{link} for double integrals, we obtain \eqref{eq6}. Combining \eqref{eq5} with \eqref{eq7} gives \eqref{eq8}. Similarly, \eqref{eq6} together with \eqref{eq10} gives \eqref{eq9}.
\end{proof}
We note that Lemma \ref{lemma17} holds more generally for hyperelliptic curves of arbitrary genus and odd differential forms, with negation replaced by the hyperelliptic involution. 
Our implementation of these algorithms is available on Github at \cite{git:repo}.

\section{Examples}\label{sec:examples}
In this section, we give explicit computations of the depth 3 locus for elliptic curves of ranks 0, 1, and 2.   Let $E$ be an elliptic curve over $\Q$, let $\mathcal{E}$ be its minimal model and let $\mathcal{X} = \mathcal{E}-\{ 0 \}$. Let $p$ be a prime of good reduction and consider the following Coleman integrals on $E(\Q_p)$:
\begin{align*}
 j_1(z) := f_1(z) &= \int_0^z \omega_0,\\
j_2(z) := f_2(z) &= \int^z \omega_0\omega_1\\
 f_3(z) &=  -\int^z \omega _0 \omega _0 \omega _1-\left(\frac{1}{2}+\frac{1}{12}\log _p (\Delta )\right)\int_0^z \omega _0 ,\\
  f_4(z) &= \int^z \omega_0\omega_1\omega_1 - \frac{1}{2}\int^z \omega_1.
\end{align*}

\subsection{Rank 0}
Let $E$ be an elliptic curve of rank zero with Tamagawa number 1 at every prime of split multiplicative reduction. Suppose $\dim H^1 _f (\Q ,V_p (E)(1))=1$. Then using the map $\Psi _3$, we find that for any $z_1 ,z_2 \in E(\Q )$, we have
\[
\det \left( \begin{array}{cc} f_3 (z_1 ) & f_3 (z_2 ) \\ f_4 (z_1 ) & f_4 (z_2 ) \end{array} \right) =0.
\]
We denote the constant $f_3 (z)/f_4 (z)$ by $c$. Note that this constant can only be computed at a point $z$ of odd order, as the functions $f_3 $ and $f_4$ are odd. For any set of primes $S$, the set $X(\Z _p )_{S,3} $ is contained in the set
\begin{equation}\label{eq:set_cub_rk_0}
\{z\in \mathcal{X}(\Z_p): f_1(z)= 0, f_2(z)\in T,f_3(z)-cf_4 (z) = 0\}
\end{equation}
where $T$ is the finite set $\{ -\sum _{\ell \neq p}\gamma _{\ell }:\gamma _{\ell }\in h_{\ell }(X(\Z _{\ell })) \}$. 

\begin{example}[Some mock rational depth 2 points persist in depth 3]\label{ex:rank0a}
Consider the elliptic curve with LMFDB label \href{http://www.lmfdb.org/EllipticCurve/Q/36/a/4}{36.a4}, which has minimal model $E:  y^2=x^3+1,$
and additive reduction at the bad primes 2 and 3. We fix the prime $p=7$.  This curve appears in \cite[Table 2]{Bia20} as an elliptic curve for which the depth 2, or quadratic Chabauty, locus
\begin{equation*}
 \{z\in \mathcal{X}(\Z_p): j_1(z) = 0, j_2(z) \in\{0, -1/2\log(3),
-2/3\log(2), -1/2\log(3) - 2/3\log(2)\} \}   
\end{equation*} 
contains algebraic \emph{mock rational points}, i.e., algebraic points that are not in $\mathcal{X}(\Z)$:
\begin{equation*}
w^{\pm}\colonequals \left(\frac{1\pm \sqrt{-3}}{2},0\right), \qquad  \pm z^{\pm}\colonequals \pm (-1\pm \sqrt{-3},3).
\end{equation*}

We investigate the depth 3 locus. First we compute the appropriate $7$-adic $L$-value:  $\mathcal{L}_7(E, 7)  = 5 + 4 \cdot 7^{2} + 4 \cdot 7^{4} + O(7^{5}) \neq 0$.

Since $E(\Q) \cong \Z/6\Z$ we use points of $E(\Q)$ to compute $c$. We find that
\begin{equation*}
c = 2 \cdot 7 + 5 \cdot 7^{2} + 3 \cdot 7^{3} + 2 \cdot 7^{4} + 4 \cdot 7^{6} + 7^{7} + 6 \cdot 7^{8} + 4 \cdot 7^{9} + O(7^{10}).
\end{equation*} 

As explained above, a point of order two is a zero of $f_3(z)$, as well as of $f_4(z)$. Therefore, we see that
\begin{equation*}
j_3(w^{\pm}) = 0,
\end{equation*}
and hence the depth 3 set \eqref{eq:set_cub_rk_0} still contains algebraic mock rational points, i.e., is larger than $\mathcal{X}(\Z)$. Its only extra points are $w^{\pm}$, since for each $z\in \{\pm z^{\pm}\}$, $j_3(z)\neq 0$. Details for this example can be found in the file \texttt{rank0ex.m} of \cite{git:repo}.

\end{example}

\subsection{Rank 1}
Assume now that $E$ has rank $1$ and Tamagawa number $1$ at every prime. Then using the map $\Psi _3$, we find that for any points $z_1 ,z_2 ,z_3 ,z_4$ in $X(\Z )$, we have
\[
\det \left( \begin{array}{cccc} f_1 (z_1 )^3 & f_1 (z_2 )^3 & f_1 (z_3 )^3 & f_1 (z_4 )^3 \\ f_1 (z_1 ) & f_1 (z_2 ) & f_1 (z_3 ) & f_1 (z_4 ) \\ f_3 (z_1 ) & f_3 (z_2 ) & f_3 (z_3 ) & f_3 (z_4 ) \\  f_4 (z_1 ) & f_4 (z_2 ) & f_4 (z_3 ) & f_4 (z_4 ) \end{array} \right) =0.
\]
Here we have used the vanishing of $\det \left( \begin{array}{cc} f_1 (z_i )^2 & f_1 (z_j )^2 \\ f_2 (z_i ) & f_2 (z_j ) \end{array} \right)$ to simplify the expression. Given an integral point $z_1 $ of infinite order, and integral points $z_2 $ and $z_3$ such that 
\[
\rk \left( \begin{array}{ccc} f_1 (z_1 )^3 & f_1 (z_2 )^3 & f_1 (z_3 )^3 \\ f_1 (z_1 ) & f_1 (z_2 ) & f_1 (z_3 ) \\ f_3 (z_1 ) & f_3 (z_2 ) & f_3 (z_3 )  \\  f_4 (z_1 ) & f_4 (z_2 ) & f_4 (z_3 )  \end{array} \right) =3,
\]
we obtain the nonzero Coleman functions 
\begin{align*}
F_2 (z) & =\det \left( \begin{array}{cc} f_1 (z_1 )^2 & f_2 (z_1) \\ f_1 (z )^2 & f_2 (z ) \end{array} \right),\\
F_3 (z) & = \det \left( \begin{array}{cccc} f_1 (z_1 )^3 & f_1 (z_2 )^3 & f_1 (z_3 )^3 & f_1 (z )^3 \\ f_1 (z_1 ) & f_1 (z_2 ) & f_1 (z_3 ) & f_1 (z) \\ f_3 (z_1 ) & f_3 (z_2 ) & f_3 (z_3 ) & f_3 (z ) \\  f_4 (z_1 ) & f_4 (z_2 ) & f_4 (z_3 ) & f_4 (z ) \end{array} \right)
\end{align*}
vanishing on $X(\Z )$.

\begin{example}
Consider the elliptic curve with LMFDB label \href{https://www.lmfdb.org/EllipticCurve/Q/37a1/}{37.a1}, which has minimal model $y^2 + y = x^3 - x$ and let $p = 7$. The zeros of $F_2$ up to negation on the short Weierstrass model $E: y^2 = x^3 - 16x + 16$ are given by the following table:

\smallskip
\begin{center}
 \begin{tabular}{|| c |  r  | c   ||}
    \hline
$E(\F_7)$ & recovered $x(z)$ & $z \in E(\Q)$  \\
  \hline
   $\overline{(4,4)}$ & $4 + 5 \cdot 7 + 4 \cdot 7^{2} + 5 \cdot 7^{3} + 2 \cdot 7^{4} + O(7^{6})$   &  --   \\ 
   & $4 + O(7^{6})$  & $(4,4)$  \\ 
 $\overline{(0,4)}$ & $5 \cdot 7 + 5 \cdot 7^{2} + 5 \cdot 7^{3} + 5 \cdot 7^{4} + 2 \cdot 7^{5} + O(7^{6})$  &   --  \\ 
 & $O(7^{6})$ & $(0,4)$  \\
$\overline{(1,6)}$ &$1 + 6 \cdot 7 + 5 \cdot 7^{2} + 6 \cdot 7^{3} + 7^{4} + 5 \cdot 7^{5} + O(7^{6})$  & --  \\ 
 & $1 + 7 + O(7^{6})$ & $(8,20)$ \\ 
 $\overline{(3,4)}$ &$3 + 3 \cdot 7 + O(7^{6})$   & $(24,116)$ \\ 
& $3 + 6 \cdot 7 + 6 \cdot 7^{2} + 6 \cdot 7^{3} + 6 \cdot 7^{4} + 6 \cdot 7^{5} + O(7^{6}) $ & $(-4,4)$   \\ 
 \hline
 \end{tabular}\end{center}
\smallskip

  Note in particular from the above table that the quadratic Chabauty locus contains three mock rational points (up to negation) that do not appear to be algebraic. We now consider the depth 3 locus and check if these points persist. First we compute the 7-adic $L$-value: $\mathcal{L}_7(E, 7)  = 4 \cdot 7 + 6 \cdot 7^{3} + O(7^{5}) \neq 0$.  
  
  We find that the zero set of $F_3(z)$ consists of $\mathcal{X}(\Z)$, together with the points with $x$-coordinate
 \begin{equation*}
3 + 3 \cdot 7 + 3 \cdot 7^{2} + 5 \cdot 7^{3} + 6 \cdot 7^{5} + O(7^{6}).
 \end{equation*}
 We conclude that 
 \begin{equation*}
 \{z\in \mathcal{X}(\Z_p): F_2(z) = 0 = F_3(z)\} = \mathcal{X}(\Z),
 \end{equation*}
that is, no mock rational points persist in depth 3 and Kim's conjecture is satisfied for this rank 1 elliptic curve.  Details for this example can be found in the file \texttt{rank1ex.m} of \cite{git:repo}.
\end{example}

\subsection{Rank 2}\label{S:rk2}
Let $E$ be an elliptic curve of Mordell--Weil rank 2 and suppose that the Tamagawa number of $E$ is 1 at every prime. Suppose that $\dim H^1 _f (\Q ,V(1))=1$. Then recall from Lemma \ref{lemma:main} that we can compute a finite set containing $X(\Z _p )_3 $ as the zeroes of $\Res (\det H_1 (P_1 ,P_2 ,P_3 ,z),\det H_2 (P_1 ,P_2 ,P_3 ,P_4 ,P_5 ,P_6 ,P_7 ,z))=0$, where the $P_i$ are integral points chosen so that the matrices $H_1 (P_1 ,P_2 ,P_3 )$ and $H_2 (P_1 ,\ldots ,P_7 )$ have full rank.

\begin{example}\label{eg:389a1}
Consider the elliptic curve with LMFDB label \href{https://www.lmfdb.org/EllipticCurve/Q/389/a/1}{389.a1}, with minimal model $y^2+y=x^3+x^2-2x$ and short Weierstrass model $E: y^2 = x^3 - 3024x + 46224.$ Modulo negation, it has the integral points
\begin{align*}
&P_1 = (120, 1188), P_2 =(12,108), P_3 = (-60,-108), P_4 = (48,108),\\
& P_5 = (-24,324), P_6 = (156,1836), P_7 = (6780,-558252), \\
& P_8 = (1416, 53244), P_9 = (228,-3348), P_{10} = (4800, 332532).
\end{align*}

We take $p=5$ as our fixed prime of good reduction.  We compute the 5-adic $L$-value: $\mathcal{L}_5(E, 5) = 4 \cdot 5^{2} + 5^{3} + 4 \cdot 5^{4} + O(5^{5}) \neq 0.$
We use the seven points $P_2, P_3,P_4, P_5, P_6, P_7, P_9$ to set constants  and  compute the Macaulay resultant of $\det (H_1 )$ and $\det (H_2 )$ and its zeros in $\mathcal{X}(\Z_p)$. The resulting zero set consists of $P_1,\dots,P_{10}$, together with the following five $p$-adic points and their additive inverses:

\small
\begin{align*}
&\left(2 \cdot 5^{2} + 5^{3} + 4 \cdot 5^{4} + 3 \cdot 5^{5} + 2 \cdot 5^{6} + O(5^{7}), 2 + 5 + 5^{2} + 2 \cdot 5^{3} + 5^{4} + 4 \cdot 5^{5} + 3 \cdot 5^{6} + O(5^{7})\right),\\
&\left(1 + 5 + 3 \cdot 5^{2} + 3 \cdot 5^{3} + 3 \cdot 5^{4} + 2 \cdot 5^{5} + 4 \cdot 5^{6} + O(5^{7}), 1 + 2 \cdot 5 + 2 \cdot 5^{2} + 2 \cdot 5^{3} + 4 \cdot 5^{5} + 4 \cdot 5^{6} + O(5^{7})\right),\\
&\left(2 + 3 \cdot 5 + 5^{2} + 5^{3} + 4 \cdot 5^{4} + 2 \cdot 5^{5} + 2 \cdot 5^{6} + O(5^{7}), 2 + 3 \cdot 5^{2} + 3 \cdot 5^{3} + 2 \cdot 5^{4} + 3 \cdot 5^{5} + 3 \cdot 5^{6} + O(5^{7})\right),\\
&\left(3 + 4 \cdot 5^{2} + 5^{3} + 2 \cdot 5^{5} + 2 \cdot 5^{6} + O(5^{7}), 2 + 5^{2} + 3 \cdot 5^{3} + 4 \cdot 5^{4} + 2 \cdot 5^{5} + 5^{6} + O(5^{7})\right),\\
&\left(3 + 2 \cdot 5 + 5^{2} + 3 \cdot 5^{4} + 3 \cdot 5^{5} + 3 \cdot 5^{6} + O(5^{7}), 2 + 4 \cdot 5 + 2 \cdot 5^{2} + 2 \cdot 5^{3} + 4 \cdot 5^{4} + 4 \cdot 5^{5} + O(5^{7})\right).
\end{align*}
\normalsize Details for this example can be found in the file \texttt{rank2ex389a1.m} of \cite{git:repo}.

\end{example}
The presence of these mock rational points in the depth 3 locus in rank 2 is unsurprising, since we have only one function cutting out the depth 3 locus. However, it may be the case that the depth 3 locus in rank 2 coincides exactly with the integral points, as we see in Example \ref{sharp}. Before we discuss that example, we describe how our work in rank 2 relates to the work of Goncharov--Levin \cite{goncharov-levin}.

\subsection{Comparison with Goncharov--Levin}\label{subsec:GL}
We recall the discussion of the Goncharov--Levin theorem in Section \ref{S:GL} and make a few further observations.  

First,  in practice, if the support of $D$ is contained in $E(\Q)$ and the Tamagawa number at a prime of split multiplicative reduction is $1$, then we may ignore the integrality condition.  If the support of $D$ is in $E(\Q)$, we may also ignore \eqref{it:local_height} at the archimedean places.  Furthermore, if the support of $D$ is in $\mathcal{X}(\Z)$ and we assume that the Tamagawa number is trivial at every prime, then condition \ref{it:local_height} is vacuous at a prime of good reduction, and at each prime of bad reduction, we may factor out the local heights to obtain
\begin{equation*}
\sum_j n_j P_j = 0 \quad \text{in}\quad E(\Q)\otimes \Q.
\end{equation*}

Thus, if we find a divisor $D = \sum n_j P_j$ with $n_j$ in $\Z$ such that
\begin{itemize}
\item $P_j\in \mathcal{X}(\Z)$;
\item $\sum n_j P_j \in E(\Q)_{\tors}$
\item $\sum n_j P_j\otimes P_j\otimes P_j = 0$ in $\sym^3 (E(\Q)\otimes \Q)$,
\end{itemize}
then conditions \eqref{it:sym}--\eqref{it:integrality} are satisfied.

Now if we furthermore specialise to the hypotheses of Section \ref{S:rk2}, where $E$ is a rank $2$ elliptic curve with generators $z_1$ and $z_2$ for the free part of the Mordell--Weil group, then we may produce candidate divisors $D$ as follows. Each $P_j$ as above must satisfy
\begin{equation*}
P_j - (a_jz_1 + b_j z_2) \in E(\Q)_{\tors}\qquad \text{for some } a_j,b_j\in \Z.  
\end{equation*}
Thus 
\begin{equation*}
\sum n_j P_j \in E(\Q)_{\tors} \qquad \text{if and only if}\qquad \sum_j n_j a_j = 0 \quad \text{and}\quad \sum n_j b_j = 0.  
\end{equation*}
Furthermore,
\begin{equation*}
\sum n_j P_j\otimes P_j\otimes P_j = 0 \in \sym^3(E(\Q)\otimes \Q) \quad \text{if and only if} \quad \sum n_j (a_j x + b_j y)^3 = 0 \in \Z[x,y]
\end{equation*}
and equating coefficients gives:
\begin{equation*}
\sum n_j a_j^3 = 0 \qquad \sum n_j a_j^2 b_j = 0, \qquad \sum n_j a_jb_j^2 = 0 \qquad \sum n_j b_j^3 = 0. 
\end{equation*}

\subsection{Example  \ref{eg:389a1} (389.a1) revisited}
If we pick $P_1,\dots P_6\in \mathcal{X}(\Z)$ such that
\begin{equation*}
M = \begin{bmatrix}
a_1 & \cdots &a_6\\
b_1 & \cdots &b_6\\
a_1^3 & \cdots& a_6^3\\
a_1^2b_1 & \cdots &a_6^2b_6\\
a_1 b_1^2 & \cdots &a_6b_6^2\\
b_1^3 & \cdots & b_6^3
\end{bmatrix}
\end{equation*}
has full rank, then for every $z\in \mathcal{X}(\Z)$ there exists $n_1,\dots n_6\in \Q$ such that $z - \sum_{j} n_j P_j$ satisfies the above conditions.  Picking $P_i$ as above, we find the candidate divisors
\begin{align*}
D_6 &= Q_6 +Q_1 - 3Q_2-2Q_3 +3Q_4 +Q_5,\\
 D_8 &= Q_8 +8Q_1+8Q_3 +8Q_4 -Q_5 +Q_7\\
D_9 &= Q_9 - 3Q_1 -15Q_2 -18 Q_3 +17Q_4 + Q_5 +4Q_7,\\ D_{10} &= Q_{10}-6Q_1 +10Q_2+20Q_3-15Q_4 -10Q_5.
\end{align*}

Let $$q(D) := \mathscr{L}(D)\frac{\pi}{L(E,2)},$$ where $\mathscr{L}(D)$ denotes the elliptic dilogarithm, extended to divisors, and $L(E,2)$ denotes the special value of the $L$-function of $E$ at $s = 2$.

We compute in PARI/GP (see the file \texttt{ellipticdilog.ipynb} in \cite{git:repo}) that 
\begin{align*}
q(D_6) &\approx  -48.62499999999999 \\ 
q(D_8) &\approx  -5.896303512293747 \;\texttt{E-20}\\ 
q(D_9) &\approx -243.1249999999999 \\
q(D_{10}) &\approx \;243.12499999999999 
\end{align*}
In particular, we numerically observe
\begin{equation}
\label{eq:quot_GL}
\frac{q(D_6)}{q(D_9)} \approx 0.2 = -  \frac{q(D_6)}{q(D_{10})}.
\end{equation}

The divisors $D_i$ defines elements of $H^1 _f (\Q ,V(1))$ under the map $\Psi _3 $, and hence the constant $c$ of the previous section is determined by the condition that $f_3(D) - cf_4(D)$ should vanish on a divisor satisfying the conditions \eqref{it:sym}--\eqref{it:integrality}.  Thus for each of $D\in \{D_6, D_8,D_9, D_{10}\}$ we should find that $\frac{f_3(D)}{f_4(D)} = c$, provided that $f_4(D)\neq 0$.  We compute that
\begin{equation*}
f_3(D_8) = f_4(D_8) = O(5^{12})
\end{equation*}
while $f_3$ and $f_4$ are both non-zero at each of $D_6, D_8,D_9$ and indeed the ratio on these divisors agrees with the value of $c$ that we computed earlier.

 In detail:
 \begin{equation*}
 f_3(z) - \sum_j n_j f_3(P_j) = cf_4(z) - c\sum_j n_j f_4(P_j) 
 \end{equation*}
 so
 \begin{equation*}
 f_3(z) - cf_4(z) =  \sum_j n_j (f_3(P_j) - cf_4(P_j)).
 \end{equation*}
 
 Moreover, we have
 \begin{align*}
 \frac{f_3(D_6)}{f_3(D_9)} = 5^{-1} + O(5^7) = 
-\frac{f_3(D_6)}{f_3(D_{10})}\\
 \frac{f_4(D_6)}{f_4(D_9)} = 5^{-1} + O(5^7)=
-\frac{f_4(D_6)}{f_4(D_{10})}
 \end{align*}
 (cf.\ \eqref{eq:quot_GL}).

 \subsection{A sharp rank 2 example} We conclude with an example of a rank 2 elliptic curve where the depth 3 locus coincides with the set of integral points.
 \begin{example}[Example: 433.a1]\label{sharp}
Consider the rank 2 elliptic curve $E$ with LMFDB label \href{https://www.lmfdb.org/EllipticCurve/Q/433/a/1}{433.a1}, which has short Weierstrass model $y^2 = x^3 + x^2 + 64$ and let $p = 3$. 

Consider also a minimal Weierstrass model of $E$, given by $y^2 + xy = x^3 + 1$, and we take $Q_1 =(-1,1)$ and $Q_2 = (0,-1)$ generators of $E$, which are integral on this minimal model. We use a supply of known rational points on $E$, written as a linear combination of $Q_1, Q_2$, and a resultant  as before to generate the function vanishing on integer points on a minimal model of $E$. 

In the residue disk of $(0,1)$, taking $(0,1)$ as basepoint, we compute a $3$-adic series which has roots $t = 0, 4$, corresponding to the integer points with $x$-coordinates $0$ and 12 in the disk, or the integral points $(0,-8)$ and $(12, -44)$. There are no other roots in this disk. Repeating this computation in other residue disks, we find that we precisely recover the integer points
   $$ (0, -8 ), (12, -44), (256, 4104), (256, -4104), (4, 12), (4, -12),$$
   $$ (476, 10396), (20, -92), (32, 184), (-4, 4) $$ up to negation (except in the mod 3 Weierstrass residue disk corresponding to $(0,1)$, where both points are recovered) and no extra 3-adic points. Details for this example can be found in the file \texttt{rank2ex433a1.m} of \cite{git:repo}. 
   
This is the first rank 2 elliptic curve example where Kim's conjecture has been verified.
\end{example}

\bibliography{biblio}
\bibliographystyle{amsalpha} 
\end{document}